\documentclass[a4paper,10pt]{amsart}

\usepackage[
  margin=30mm,
  marginparwidth=25mm,     % + <- Width of your marginpar
  marginparsep=2mm,       % + <- Gap between text block and marginpar
  bottom=25mm,
  ]{geometry}

\usepackage[bbgreekl]{mathbbol}
\usepackage{amsfonts}
\usepackage{latexsym,amssymb,amsthm,mathrsfs,amsmath,amscd,enumerate,enumitem, amscd,color}

\usepackage{tikz}

\DeclareSymbolFontAlphabet{\mathbb}{AMSb}
\DeclareSymbolFontAlphabet{\mathbbl}{bbold}

\newtheorem{ThA}{Theorem}

\newtheorem{thm}{Theorem}[section]

\theoremstyle{definition}

 \theoremstyle{remark}

\usepackage{hyperref}

\newcommand{\N}{\mathbb{N}}
\newcommand{\Z}{\mathbb{Z}}
\newcommand{\LL}{\mathcal{L}}

\newcommand{\R}{\mathbb{R}}

\newcommand{\edproof}{ $\hfill {\Box}$}
\newcommand{\black}{\textcolor{black}}

\numberwithin{equation}{section}
\allowdisplaybreaks

\begin{document}

\footnotetext{Last modification: \today.}

\title[]{Extension theorems for logarithmic Schr\"odinger and discrete Laplacian operators}

\author[J. J. Betancor]{J. J. Betancor}

\author[M. de Le\'on-Contreras]{M. de Le\'on-Contreras}

\author[L. Rodr\'{\i}guez-Mesa] {L. Rodr\'{\i}guez-Mesa}

\address{Jorge J. Betancor, Marta de Le\'on-Contreras, Lourdes Rodr\'iguez-Mesa\newline
	Departamento de An\'alisis Matem\'atico, Universidad de La Laguna,\newline
	Campus de Anchieta, Avda. Astrof\'isico S\'anchez, s/n,\newline
	38721 La Laguna (Sta. Cruz de Tenerife), Spain}
\email{jbetanco@ull.es, 
mleoncon@ull.edu.es,
lrguez@ull.edu.es}

\thanks{The authors are partially supported by the Grant  PID2023-148028NB-I00 funded by MICIU/AEI/10.13039/501100011033 and by ERDF/EU”.
}



%\date{}

%%% ----------------------------------------------------------------------

\begin{abstract}
 {In this paper we consider logarithmic operators in two different contexts: the  adapted to (continuous) Schr\"odinger operators and the classical discrete setting. The Schr\"odinger operator $\mathcal L_V$ on $\mathbb R^d$ is defined as $\mathcal L_V=-\Delta+V$, where the potential $V$ is nonnegative and satisfies a reverse H\"older inequality and, as usual, $\Delta$ denotes the Euclidean Laplacian, while the discrete Laplacian $\Delta_d$ on $\mathbb Z$ is given by $(\Delta_df)(n)=f(n+1)-2f(n)+f(n-1)$, $n\in \mathbb Z$. Both logarithmic operators $\log \mathcal L_V$ and $\log (-\Delta_d)$ are nonlocal operators and we will define them through suitable extension problems. The extension problems for logarithmic operators are inspired by the one introduced by  Caffarelli and Silvestre for the fractional Laplacian but, in this case, the logarithmic operators are obtained as the boundary values of the extension in a more involved way.}

%%% VERSION ANTERIOR%%%%%%%  In this paper we consider logarithmic operators in two different settings. The Schr\"odinger operator $\mathcal L_V$ in $\mathbb R^d$ is defined by $\mathcal L_V=-\Delta+V$, where the potential $V$ is nonnegative and as usual $\Delta$ represents the Euclidean Laplacian. We define the discrete Laplacian $\Delta_d$ on $\mathbb Z$ by $(\Delta_df)(n)=f(n+1)-2f(n)+f(n-1)$, $n\in \mathbb Z$. The logarithmic operators $\log \mathcal L_V$ and $\log \Delta_d$ are nonlocal operators and we will define them through extension problems. The extension problem for logarithmic operators is inspired in the one of Caffarelli and Silvestre for the fractional Laplacian, but in this case the logarithmic operators are obtained as the boundary value of the extension in a more involved way.
\end{abstract}

%%% ----------------------------------------------------------------------
\maketitle
%%% ----------------------------------------------------------------------

\section{Introduction}
In this paper we show that logarithmic operators associated with Schrödinger operators on $\mathbb{R}^d$ and with the discrete Laplacian on $\mathbb{Z}$ can be obtained through suitable extension problems.  These logarithmic operators are of non-local nature. The extension problem associated with the logarithmic Laplacian on $\mathbb R^d$ has been recently introduced in \cite{ChHW}. Although this extension problem was inspired by the one developed by Caffarelli and Silvestre (see \cite{CS}) for the fractional Laplacian, the logarithmic operator appears as the boundary values of the solution to the extension problem in a more involved way.

The celebrated Caffarelli-Silvestre extension theorem can be stated as follows.
\begin{ThA}
Let $\sigma \in (0,1)$.   We consider the Dirichlet problem
$$
\left\{ \begin{array}{cc}
     \nabla \cdot(t^{1-2\sigma}\nabla u(x,t))=0,& (x,t)\in\R^d\times(0,\infty)  \\
    u(0,x)=f(x), & x\in\R^d,
\end{array}\right.
$$
where $f\in\mathcal{S}(\R^d)$, the space of Schwartz functions, and $\nabla=(\partial_{x_1}, \dots,\partial_{x_n},\partial_t)$, and the associated Dirichlet to Neumann map, given by $f\in\mathcal{S}(\R^d) \mapsto  -\lim_{t\to 0^+}t^{1-2\sigma}\partial_t u(\cdot, t).$ Then,  the last mapping coincides with $\frac{\Gamma(1-\sigma)}{2^{2\sigma-1}}(-\Delta)^\sigma$, where $(-\Delta)^\sigma$ is the $\sigma-$fractional Laplacian.

\end{ThA}
The Caffarelli-Silvestre extension has opened numerous and new lines of research, one of which focuses on identifying operators that admit a representation through an extension problem. In this direction, Kwa\'snicki and Mucha (see \cite{KM}) proved that if $\phi$ is a complete Bernstein function, the operator $\phi(-\Delta)$ can be obtained from an extension problem. The results in \cite{KM}  were extended to other differential operators in \cite{AH} by using \^Ito calculus. Extension problems have been studied in  more abstract settings by Stinga and Torrea (see \cite{ST}), and Galé, Miana and Stinga (see \cite{GMS}). Moreover, the fractional Laplace Beltrami operator on some noncompact manifolds has been defined through an extension problem in \cite{BGS} and \cite{BP}. Arendt, ter Elst and Warma (\cite{AEW})  considered the problem for sectorial operators in Hilbert spaces.

The logarithmic Laplacian operator, $\log(-\Delta),$ was studied in \cite{ChW}, where a pointwise representation was obtained (see \cite[Theorem 1.1]{ChW}). Chen and Véron (\cite{ChV1}) analyzed the Cauchy problem associated with $\log(-\Delta),$ and proved the existence of the fundamental solution of the logarithmic Laplacian  in dimension $d\ge 3$. Recently, Lee (\cite{Lee}) extended the result about the fundamental solution of $\log(-\Delta)$  by using an approach via the division problem. On the other hand, spectral properties for the logarithmic Laplacian have been studied in  \cite{ChV2}, \cite{ChW}, and \cite{JSW} while the $m-$order logarithmic Laplacian was studied in \cite{ChH}.

Logarithmic operators have also been defined in several other settings. Logarithmic Bessel operators were analyzed in \cite{FLZ}, while Fall and Felli (see  \cite{FF2} and \cite{FF1}) proved unique continuation properties and essential self-adjointness of the relativistic Schrödinger operator with a singular homogeneous potential defined by
$$
H=(-\Delta+m^2)^s-\frac{a(x/|x|)}{|x|^{2s}}-h(x), \quad m \ge 0,
$$
where the pointwise representation for $(-\Delta+m^2)^s$, $0<s<1$, can be found in \cite[(1.3)]{FF1}. 

On the other hand, Feulefack (see \cite{F1,F2}) introduced and studied the logarithm for the operator $I+(-\Delta)^s$, $s\in(0,1]$. Until very recently there were no results of the logarithmic Laplacian on manifolds. In \cite{Pr}, the logarithmic operator $\log (-\Delta+mI)$, for $m>1$, was defined on closed manifolds, while Chen (\cite{Ch}) defined $\log(-\Delta)$ on general Riemannian manifolds, and Chen and Xu (see \cite{ChX}) on general graphs. 

The logarithmic Schrödinger operator, $\log(-\Delta+V)$,  where $V$ is a nonnegative potential satisfying a reverse Hölder inequality, $RH_q$, with $q>d/2$, has been defined recently by Betancor, Dalmasso, Fariña and Quijano in \cite{BDFQ}. Here we summarize the main points in order to state our results, see Section \ref{sec2} for more details.

The Schrödinger operator, $\LL_V=-\Delta+V$, can be defined through the following sesquilinear form
$$
T(f,g)=\int_{\R^d}\nabla f\;\overline{\nabla g}\;dx+\int_{\R^d}V\;f\;\overline{g}\;dx,
$$
for $f,g\in D_T:=\{h\in L^2(\R^d):\: \nabla h\in L^2(\R^d) \text{ and } \sqrt{V} h\in  L^2(\R^d) \}.$

The sesquilinear form $T$ is closed and nonnegative and the domain $D_T$ is dense in $L^2(\R^d)$. Therefore, there exists a selfadjoint operator $\LL_V:  D(\LL_V)=D_T\subset L^2(\R^d)\to L^2(\R^d)$ such that $\langle\LL_Vf,g \rangle=T(f,g)$, $f,g\in D(\LL_V)$ (see \cite[Theorem VIII.15]{RS}). The space $C^\infty_c(\R^n)$ of smooth functions with compact support in $\R^d$ is contained in $D(\LL_V)$
 and $\LL_Vf=-\Delta f+V f$, for every $f\in C^\infty_c(\R^n)$. 

 Hence, there exists a spectral measure $E_V$ supported in the spectrum $\sigma(\LL_V)$ of $\LL_V$ such that
 $$
 \LL_Vf=\int_{[0,\infty)}\lambda\; dE_V(\lambda)f, \quad f\in  D(\LL_V),
 $$
 and $ {D(\LL_V)}=\left\{f\in L^2(\R^d):\:\int_0^\infty \lambda^2\; d\mu_{f,f}^V(\lambda)<\infty\right\}.$ Here, for every $f,g\in  L^2(\R^d),$ $\mu_{f,g}^V(U)=\langle E_V(U)f,g\rangle$, for every Borel subset $U$ of $\R.$ Notice that $ E_V(\{0\})=0$, because $0$ is not an eigenvalue of $\LL_V$.

 Thus, the logarithmic $\log (\LL_V)$ of $\LL_V$ is defined by
 $$
 \log(\LL_V)f=\int_0^\infty \log\lambda \;dE_V(\lambda)f, $$
 for every $f\in D(\log(\LL_V))=\left\{f\in L^2(\R^d):\:\int_0^\infty  {|\log\lambda|^2}\; d\mu_{f,f}^V(\lambda)<\infty\right\}.$

 In \cite[Theorem 1.1 (a)]{BDFQ}, it was established that if $f\in D(\LL_V^{s_0})\cap D(\log(\LL_V)) $, for some $s_0\in (0,1],$ then
 $$
 \log(\LL_V)f=\lim_{s\to 0^+}\frac{\LL_V^{s}f-f}{s},
 $$
where $\LL_V^{s}f=\int_0^\infty\lambda^s \; dE_V(\lambda)f,$ $f\in D(\LL_V^{s}):=\left\{f\in L^2(\R^d):\:\int_0^\infty \lambda^{2s}\; d\mu_{f,f}^V(\lambda)<\infty\right\},$ $s\in(0,1].$

Thus, for every $t>0$  we can define $T_t^Vf=\int_0^\infty e^{-\lambda t} dE_V(\lambda)f$, $f\in L^2(\R^d)$, so that the family $\{ T_t^V\}_{t>0}$ is the $C_0$-semigroup in $L^2(\R^d)$ generated by $\LL_V.$ Then, for every $t>0$, there exists a measurable function $T_t^V:\R^d\times\R^d\to\R$ such that 
\begin{equation}\label{eq1.1}
    T_t^V(f)(x)=\int_{\R^d}T_t^V(x,y)f(y)dy, \quad f\in  L^2(\R^d), \quad x\in\R^d.
\end{equation}
Furthermore, by using the integral representation \eqref{eq1.1},  each $T_t^V$ can be extended as a contraction in $L^p(\R^d)$, so that $\{ T_t^V\}_{t>0}$ is a semigroup of contractions in $L^p(\R^d)$, for $1\le p\le \infty$. However, $\{ T_t^V\}_{t>0}$ is not Markovian, that is, $T_t^V1\neq 1$, $t>0$. This fact leads to essential differences between the proofs of the results concerning the Schrödinger setting and those for the Laplacian.

The following pointwise representation of $ \log(\LL_V)$ was established in \cite[Theorem 1.2]{BDFQ}.

\begin{ThA}\label{ThB}
    Let $d\ge 3$ and $q>d/2$. Suppose that $V\in RH_q$ and $f\in C^\infty_c(\R^d).$ Then,
\begin{align}\label{eq1.6}
((\log \LL_V)f)(x)=&-\int_{B(x,1)}(f(y)-f(x))\int_0^\infty \frac{T_t^V(x,y)}{t} dt dy\nonumber\\
&-\int_{\R^d\setminus B(x,1)}f(y)\int_0^\infty\frac{T_t^V(x,y)}{t} dt dy-f(x)K(x), \quad \text{ for almost all }x\in \R^d, 
\end{align}
where
\begin{align*}
K(x)&=2\log \rho(x)+\int_{\R^d}\int_0^{\rho(x)^2}\frac{T_t^V(x,y)-T_t(x-y)}{t}dydt-\int_{\R^d\setminus B(x,1)}\int_0^{\rho(x)^2}\frac{T_t^V(x,y)}{t} dydt\\
&+\int_{B(x,1)}\int_{\rho(x)^2}^\infty  \frac{T_t^V(x,y)}{t} dydt\;+\;\gamma,\quad x\in \mathbb R^d,
\end{align*}
being $\gamma$ the Euler-Mascheroni constant and $\rho$ the critical radius function.
 \end{ThA}
Note that, in constrast with the pointwise representation of $\log(-\Delta)$ established in \cite[Theorem 1.1]{ChW}, the corrector factor $K$ in Theorem \ref{ThB} is not constant. This is due to the fact that the semigroup $\{ T_t^V\}_{t>0}$ is not Markovian. In addition, observe that as a special case of Theorem \ref{ThB},  a pointwise representation of $\log(-\Delta+m^2)$ can be obtained.

On the other hand, as it was mentioned in \cite{BDFQ}, if $f\in {\rm Lip}^\theta(\R^d) {\cap L_0^1(\R^d)}$,  {where  ${\rm Lip}^\theta(\R^d)$ denotes the $\theta$-Lipschitz space on $\R^d$ with $\theta\in (0,1]$, and, for $\sigma\ge 0$},
$$
L_\sigma^1(\R^d):=\left\{f \text{ measurable in $\R^d$} : \int_{\R^d}\frac{|f(y)|}{(1+|y|)^{d+\sigma}}dy<\infty\right\},
$$ 
then
\begin{align*}
    \lim_{s\to 0^+}\frac{1}{s}\left( \frac{1}{\Gamma(-s)}\int_0^\infty \frac{T_t^V(f)(x)-f(x)}{t^{s+1}}dt {-f(x)}\right)&\\
    &\hspace{-5cm}=-\int_{B(x,1)}(f( {y})-f( {x}))\int_0^\infty \frac{T_t^V(x,y)}{t} dtdy\\
    &\hspace{-5cm}\quad -\int_{\R^d\setminus B(x,1)}f(y)\int_0^\infty\frac{T_t^V(x,y)}{t} dtdy {-}K(x)f(x), \quad x\in \R^d.
\end{align*}
According to \cite[Proposition 2.4]{BDFQ}, if $s\in (0,1)$ and $f\in D(\LL_V^s)$, then
$$
\LL_V^s(f)=\lim_{m\to\infty}\frac{1}{\Gamma(-s)}\int_{1/m}^m\frac{T_t^V(f)-f}{t^{s+1}}dt,
$$
where the integrals are understood in the $L^2((\frac{1}{m},m))-$Bochner sense, for every $m\in\N$, and the limit is understood in $L^2(\R^d).$ This property justifies to define
\begin{align*}
    ((\log \LL_V)f)(x)=&-\int_{B(x,1)}(f( {y})-f( {x}))\int_0^\infty \frac{T_t^V(x,y)}{t} dt dy\\
& -\int_{\R^d\setminus B(x,1)}f(y)\int_0^\infty\frac{T_t^V(x,y)}{t} dt dy {-}K(x)f(x),
\end{align*}
provided that $f\in  {{\rm Lip}^\theta(\R^d)\cap L_0^1(\R^d)}$, for some $\theta\in (0,1]$.

Before establishing our extension theorem for the nonlocal operator $\log \LL_V$, we introduce some definitions. We say that a function $f\in L^1_{\rm loc}(\R^d)$ has algebraic growth, in short $f\in \mathcal{AG}(\R^d)$, when there exist $C,\sigma>0$ such that
$$
\int_{B(x,1)}|f(y)|dy \le C (1+|x|)^\sigma, \quad x\in \R^d.
$$
We say that a measurable function $f:\R^d\to \R$ is Dini continuous at $x\in\R^d$ when
$$
\int_0^1\frac{w_{f,x}(r)}{r}dr<\infty,
$$
where $w_{f,x}(r)=\sup_{y\in B(x,r)}|f(y)-f(x)|,$ $r\in (0,1)$, and $f$ is uniformly Dini continuous in $\Omega\subset \R^d$ provided that
$$
\int_0^1\frac{w_{f,\Omega}(r)}{r}dr<\infty,
$$
where $w_{f,\Omega}(r)=\sup_{x\in\Omega}w_{f,x}(r),$ $r\in (0,1)$.

 \begin{thm}[Extension problem for the Schrödinger logarithmic operator]\label{Th1.1}
Let $d\ge 3$ and $V\in RH_q$, with $q>d/2.$ Suppose that $f\in  {{\rm Lip}^\theta(\R^d)\cap L_0^1(\R^d)}$, for some $\theta\in (0,1]$. We define
$$
u_f(x,t)=\frac{1}{2}\int_0^\infty T_u^V(f)(x)\frac{e^{-\frac{t^2}{4u}}}{u}du, \quad x\in\R^d, \: t>0.
$$
We have that $u_f\in C^1(\R^{d+1}_+)\cap \mathcal{AG}(\R^{d+1}_+)$ and
\begin{itemize}
    \item[(i)] $\displaystyle\lim_{t\to\infty} u_f(x,t)=0$, for every $x\in \R^d.$\\
      \item [(ii)] $\big(\partial_t^2+\frac{1}{t}\partial_t {-}\LL_V\big) u_f(x,t)=0$, for every $x\in \R^d$ and $t>0.$\\
        \item[(iii)]$\displaystyle\lim_{t\to 0^+}t\partial_t u_f(\cdot,t)=-f$ in $L^1_{\rm loc}(\R^d),$ that is, for every $R>0$,
        $$
        \lim_{t\to 0^+}\int_{B(0,R)}|t\partial_t u_f(x,t)+f(x)| dx=0.
        $$
          \item[(iv)]$\displaystyle\lim_{t\to 0^+}\frac{u_f(\cdot,t)}{\log t}=-f$ in  $L^1_{\rm loc}(\R^d).$\\
          
            \item[(v)]$((\log \LL_V)f)(x)=-{2}\displaystyle\lim_{t\to 0^+}(u_f(x,t)+f(x)\log t)-f(x)h(x)$ in the distributional sense, that is, for every $\varphi\in C^\infty_c(\R^d),$
            \begin{align}\label{eq.1.7}
                \int_{\R^d}f(x)((\log\LL_V)\varphi)(x)dx&=-{2}\lim_{t\to 0^+}\int_{\R^d}(u_f(x,t)+f(x)\log t)\varphi(x)dx\nonumber\\
                &\quad -\int_{\R^d}f(x)h(x)\varphi(x)dx,
            \end{align}
            where
\begin{align*}
h(x)&=K(x)-\int_{B(x,1)}\int_0^\infty\frac{T_u^V (x,y)-T_u(x-y)}{u}dudy-\alpha _d-\beta_d.
\end{align*}
            being $\alpha_d=2\displaystyle\int_0^1 (1+t)^{-d/2}t^{d/2-1}dt$ and $\beta_d=2\displaystyle\int_1^\infty ((r^2+1)^{-1/2}-r^{-d})r^{d-1}dr.$
\end{itemize}
Furthermore, \eqref{eq.1.7} holds for every $\varphi$ uniformly Dini continuous in $\R^d$ with compact support,  { in short, $\varphi\in C_{c,D}(\R^d)$}. Moreover, if $f$ is Dini continuous at $x\in\R^d,$ then
$$
((\log\LL_V)f)(x)= -2\lim_{t\to 0^+}\big(u_f(x,t)+f(x)\log t\big)-f(x)h(x).
$$
\end{thm}

Observe that from Theorem \ref{Th1.1} we get a solution for an extension problem associated with the operator $\log(-\Delta+m^2)$.

\vspace{0.4cm}

On the other hand, the study of analytic and  geometric properties of graphs has been an active research area (see \cite{BHLLMY,BHY, ChWY,ChuY, Chu,GLY, ZLY}). By taking as a starting point the following definition through a Bochner integral
$$
\log(-\Delta)f=\int_0^\infty \frac{e^{-t}f-e^{t\Delta}f}{t}dt
$$
that makes sense on any weighted graph provided that $f$ satisfies certain mild growth conditions,  Chen and Xu (see \cite[Theorem 1.1]{ChX}) have obtained a pointwise representation for the logarithmic Laplacian on stochastically complete weighted graphs. The analysis of the logarithmic Laplacian on graphs requires a precise control of the heat kernel under additional structural assumptions on the graph. As can be seen in the proof of \cite[Theorem 1.2]{ChHW} (and also in our proof of Theorem \ref{Th1.1}),  considering an extension problem related with the logarithmic Laplacian on general graphs is a more cumbersome question than the one for the fractional Laplacian, $(-\Delta)^s,$ $0<s<1.$

In the sequel, we shall focus on the toy case of the unweighted graph $\Z$. This is the first step to consider the problem for general weighted graphs. Harmonic analysis operators in the discrete case were studied in \cite{CGRTV} and \cite{CRSTV}. An extension problem for fractional powers $(-\Delta_d)^s$, $0<s<1$, of the discrete Laplacian defined by $\Delta_df(n)=f(n+1)-2f(n)+f(n-1)$, $n\in\Z$, can be found in \cite[Remark 1.4]{CRSTV}. As a special case of \cite[Theorem 1.1]{ChX} for $\Z$ as unweighted graph, we can  {get a pointwise representation of the logarithmic discrete Laplacian, as follows in the next result.} We shall denote by $C_c(\Z)$ the space of sequences $\{f(n)\}_{n\in\Z}$ such that the set $\{n\in\Z:\: f(n)\neq 0\}$ is finite.
 \begin{ThA}\label{ThC}
 Let $f\in C_c(\Z)$. We have that
 \begin{align}\label{eq1.8}
( \log (-\Delta_d)f)(n)&=\sum_{\substack{m\in\Z\\ m\neq n}}W_0(n-m)(f(n)-f(m))\nonumber\\
&\quad -\sum_{m\in\Z}W_\infty(n-m)f(m)-\gamma f(n), \quad n\in\Z,
 \end{align}
 where $\gamma$ denotes the Euler-Mascheroni constant, and the kernels $W_0$ and $W_\infty$ are given by
 $$
 W_0(m)=\displaystyle\int_0^1\frac{p_t(m)}{t}dt,\quad  W_\infty(m)=\displaystyle\int_1^\infty\frac{p_t(m)}{t}dt, \quad m\in\Z.
 $$
 Here, $p_t(m),$ $m\in\Z$, $t>0$, denotes the heat kernel associated with $-\Delta_d,$ given by $p_t(m)=e^{-2t}I_{|m|}(2t)$, $m\in\Z$ and $t>0$, being $I_\nu $  the modified Bessel function of first kind and order $\nu.$
\end{ThA}
\noindent The right hand side of \eqref{eq1.8} is also defined by sequences $\{f(m)\}_{m\in\Z}$ such that $\sum_{m\in\Z}\frac{|f(m)|}{\sqrt{1+|m|}}<\infty$, so we can extend the definition of $\log(-\Delta_d)$ to sequences with this property.
Now we establish our result about the extension problem concerning to $\log(-\Delta_d).$

\begin{thm}[Extension problem for the logarithmic discrete Laplacian operator]\label{Th1.2}
Let $\{f(m)\}_{m\in\Z}$ be a sequence such that  $\sum_{m\in\Z}\frac{|f(m)|}{\sqrt{1+|m|}}<\infty$. We define 
$$
u_f(n,t)=\int_0^\infty p_u(f)(n)\frac{e^{-\frac{t^2}{4u}}}{u}du, \quad n\in\Z \text{ and } t>0,
$$
where
$p_u(f)(n)=\sum_{m\in\Z}p_u(n-m)f(m),$ $n\in\Z$ and $u>0$.

Then, the following properties hold:
\begin{itemize}
    \item[(i)] $\displaystyle\lim_{t\to\infty} u_f(n,t)=0$,  $n\in \Z.$\\
      \item[(ii)] $\big(\partial^2_t+\frac{1}{t}\partial_t+\Delta_d\big) u_f(n,t)=0$, $n\in \Z$ and $t>0.$ \\      
      \item[(iii)] $\displaystyle\lim_{t\to 0^+}t\partial_t u_f(n,t)=- {2}f(n)$, $n\in\Z$.\\ 
      \item[(iv)]$\displaystyle\lim_{t\to 0^+}\frac{u_f(n,t)}{\log t}=-2f(n)$, $n\in\Z.$\\
      \item[(v)]$(\log(-\Delta_d)f)(n)=-\displaystyle\lim_{t\to 0^+} (u_f(n,t)+2 {f(n)}\log t) {+}f(n)K$, $n\in\Z,\,$ where 
            $$
            K=\displaystyle -\gamma \;+\;\int_{1/4}^\infty \frac{e^{-v}}{v}dv-\int_0^{1/4}\frac{1-e^{-v}}{v}dv.
            $$
            \end{itemize}
\end{thm}
Theorems \ref{Th1.1} and \ref{Th1.2} are proved in Sections \ref{sec2} and \ref{sec3}, respectively. Throughout this paper, $C$ and $c$ will always denote positive constants that can change in each
occurrence.

\section{The Schrödinger setting}\label{sec2}
 
 Along this section we shall consider the Schrödinger operator on $\R^d$, $d\ge 3$,  given by $\LL_V=-\Delta+V,$  where $V$ is a nonnegative potential satisfying a reverse Hölder inequality, $V\in RH_q,$ with $q>d/2$, that is,  for every ball $B\subset\R^d$ it holds that
$$
\left( \frac{1}{|B|}\int_B V(y)^q dy\right)^{1/q}\le \frac{C}{|B|}\int_B V(y)dy.$$

In this setting, the following function, so-called critical radius, plays a crucial role
$$
\rho(x)=\sup\left\{r>0:\:\frac{1}{r^{d-2}}\int_{B(x,r)}V(y)dy\le 1\right\}.
$$
From the comments in the previous section, the semigroup generated by $-\LL_V$ has the pointwise representation
\begin{equation}
    T_t^V(f)(x)=\int_{\R^d}T_t^V(x,y)f(y)dy, \quad x\in\R^d,\text{ for } f\in  L^p(\R^d), \: 1\le p\le \infty.
\end{equation}
We now recall some estimates involving the integral kernel $T_t^V(x,y)$, $x,y\in\R^d$ and $t>0$, that will be useful in the sequel.

Let us  denote by $T_t(z)$, $z\in\R^d$ and $t>0$, the Euclidean heat kernel in $\R^d$, that is, 
$$
T_t(z)=\frac{e^{-\frac{|z|^2}{4t}}}{(4\pi t)^{d/2}}, \quad z\in\R^d, \:t>0.
$$
Since  $V\ge 0$, the Feynman-Kac formula leads to
\begin{align}\label{F1}
    0\le T_t^V(x,y)\le T_t(x-y), \quad x,y\in\R^d \quad \text{and } t>0.
\end{align}
Moreover, since $V\in RH_q,$ with $q>d/2$,
according to \cite[Proposition 2.4]{DGMTZ}, for every $N\in\N$ there exist $C,c>0$ such that

\begin{equation}\label{F2}
|T_t^V (x,y)|+|t\partial _tT_t^V (x,y)|\leq C\frac{e^{-c\frac{|x-y|^2}{t}}}{t^{\frac{d}{2}}}\Big(1+\frac{\sqrt{t}}{\rho (x)}+\frac{\sqrt{t}}{\rho (y)}\Big)^{-N},\quad x,y\in \mathbb R^d \mbox{ and }t>0.
\end{equation}
In addition, the following Lipschitz regularity for the Schrödinger heat kernel holds (see \cite[Proposition 4.11]{DzZ2}): there exists $C>0$ such that
\begin{equation}\label{F3}
    |T_t^V(x,y)-T_t(x-y)|\le  {C}\left\{\begin{array}{cc}
        \left(\frac{\sqrt{t}}{\rho(x)} \right)^\delta \varphi_t(x-y), & \sqrt{t}\le \rho(x)  \\
       \varphi_t(x-y),  & \sqrt{t}>\rho(x), 
    \end{array} \right. \qquad x,y\in\R^d,
\end{equation}
where $\delta {=2-d/q}>0$ and $\varphi$ is a function in the Schwartz class $\mathcal{S}(\R^d),$  so that $\varphi_t(z)=\frac{1}{t^{d/2}}\varphi\left( \frac{z}{\sqrt{t}}\right)$, $z\in \R^d$ and $t>0.$

\subsection{Proof of Theorem \ref{Th1.1}}

Suppose that $f\in L^1_0(\mathbb R^d)$. We consider
$$
u_f(x,t)=\frac{1}{2}\int_0^\infty T_u^V (f)(x)\frac{e^{-\frac{t^2}{4u}}}{u}du,\quad x\in \mathbb R^d\mbox{ and }t>0.
$$
According to \eqref{F1} we have that
\begin{align}\label{bound}
|u_f(x,t)|&\le C\int_{\mathbb R^d}\int_0^\infty |f(y)|T_u^V (x,y)\frac{e^{-\frac{t^2}{4u}}}{u}dudy\leq C\int_{\mathbb R^d}|f(y)|\int_0^\infty \frac{e^{-\frac{t^2+|x-y|^2}{4u}}}{u^{\frac{d}{2}+1}}dudy\nonumber\\
&\leq C\int_{\mathbb R^d}\frac{|f(y)|}{(t^2+|x-y|^2)^{\frac{d}{2}}}dy\leq C\left(\int_{\mathbb |y|\leq 2|x|}+\int_{|y|> 2|x|}\right)\frac{|f(y)|}{(t+|x-y|)^d}dy\nonumber\\
&\leq C\left(\frac{(1+|x|)^d}{t^d}\int_{\mathbb R^d}\frac{|f(y)|}{(1+|y|)^d}dy+\int_{\mathbb R^d}\frac{|f(y)|}{(t+|y|)^d}dy\right)\nonumber\\
&\leq C\left(\frac{(1+|x|)^d}{t^d}+\frac{1}{\min\{1,t^d\}}\right)\int_{\mathbb R^d}\frac{|f(y)|}{(1+|y|)^d}dy<\infty,\quad x\in \mathbb R^d\mbox{ and }t>0.
\end{align}
Thus we can write
$$
u_f(x,t)=\frac{1}{2}\int_{\mathbb R^d}f(y)\int_0^\infty T_u^V (x,y)\frac{e^{-\frac{t^2}{4u}}}{u}dudy,\quad x\in \mathbb R^d\mbox{ and }t>0.
$$
Since 
$$
|u_f(x,t)|\leq C\int_{\mathbb R^d}\frac{|f(y)|}{(t^2+|x-y|^2)^\frac{d}{2}}dy,\quad x\in \mathbb R^d\mbox{ and }t>0,
$$
by proceeding as in the proof of \cite[Lemma 3.2]{ChHW} we get that $u_f(\cdot ,t)\in {L^1_\sigma (\mathbb R^d)}$, for every $t>0$ and $\sigma >0$. Furthermore, for every $\sigma >0$ there exists $C_\sigma>0$ such that
\begin{equation}\label{2.1}
\int_{\mathbb R^d}\frac{|u_f(x,t)|}{(1+|x|)^{d+\sigma}}dx\leq C_\sigma (1+\textcolor{black}{\log^-t)},\quad t>0.
\end{equation}
Here $\textcolor{black}{\log^-t}=\max\{-\textcolor{black}{\log t},0\}$, $t>0$. From \eqref{2.1} with $\sigma =1$ and the arguments in the proof of \cite[Theorem 3.1, (i)$\Longrightarrow $(ii)]{ChHW} we deduce that $u_f\in \mathcal {AG}(\mathbb R_+^{d+1})$.\\

\noindent{\bf Proof of $(i)$}. From the previous estimates we have that
$$
|u_f(x,t)|\leq C\int_{\mathbb R^d}\frac{|f(y)|}{(1+|x-y|)^d}dy\leq C(1+|x|)^d\int_{\mathbb R^d}\frac{|f(y)|}{(1+|y|)^d}dy,\quad x\in \mathbb R^d\mbox{ and }t>1.
$$
By using dominated convergence theorem we get that 
$$
\lim_{t\rightarrow \infty}u_f(x,t)=0,\quad x\in \mathbb R^d.
$$

\noindent{\bf Proof of $(ii)$}. Let $(x_0,t_0)\in \mathbb R_+^{d+1}$. According to \eqref{F1} and \textcolor{black}{by proceeding as in \eqref{bound}} we have that
\begin{align*}
 \int_{\mathbb R^d}\int_0^\infty |f(y)|\frac{e^{-\frac{|x-y|^2}{4u}}}{u^{\frac{d}{2}+1}}\Big|\textcolor{black}{\partial _t^2e^{-\frac{t^2}{4u}}+\frac{1}{t}\partial _te^{-\frac{t^2}{4u}}}\Big|dudy&\\
&\hspace{-5cm}\leq C\int_{\mathbb R^d}|f(y)|\int_0^\infty\frac{e^{-\frac{|x-y|^2+t^2}{4u}}}{u^{\frac{d}{2}+1}}\Big(\textcolor{black}{\frac{t^2}{u^2}+\frac{1}{u}}\Big)dudy\\
&\hspace{-5cm}\textcolor{black}{\leq C\int_{\mathbb R^d}|f(y)|\int_0^\infty\frac{e^{-\frac{|x-y|^2+t^2}{8u}}}{u^{\frac{d}{2}+2}}dudy=C\int_{\mathbb R^d}\frac{|f(y)|}{(|x-y|^2+t^2)^{\frac{d}{2}{+1}}}dy}\\
&\hspace{-5cm} {\le C\left(\frac{(1+|x|)^{d+2}}{t^{d+2}}+\frac{1}{\min\{1,t^{d+2}\}}\right)\int_{\mathbb R^d}\frac{|f(y)|}{(1+|y|)^{d+2}}dy}\\
&\hspace{-5cm}\textcolor{black}{\leq C\int_{\mathbb R^d}\frac{|f(y)|}{(1+|y|)^{d+2}}dy},\quad (x,t)\in B\Big((x_0,t_0),\frac{t_0}{2}\Big).
\end{align*}
\textcolor{black}{Here $C=C(x_0,t_0)$}. These estimates allow us to write
\begin{align*}
\textcolor{black}{\partial ^2_ tu_f(x,t)+\frac{1}{t}\partial _tu_f(x,t)}&= {\frac{1}{2}}\int_{\mathbb R^d}f(y)\int_0^\infty \frac{T_u^V (x,y)}{u }\left(\partial ^2_t+\frac{1}{t}\partial _t\right)e^{-\frac{t^2}{4u}}dudy\\
&= {\frac{1}{2}}\int_{\mathbb R^d}f(y)\int_0^\infty \frac{T_u^V (x,y)}{u}\Big(\frac{t^2}{4u^2}-\frac{1}{u}\Big)e^{-\frac{t^2}{4u}}dudy\\
&=\textcolor{black}{ {\frac{1}{2}}\int_{\mathbb R^d}f(y)\int_0^\infty T_u^V (x,y)\partial _u\left(\frac{1}{u}e^{-\frac{t^2}{4u}}\right)dudy},\quad (x,t)\in \mathbb R_+^{d+1}.
\end{align*}
Partial integration leads to 
$$
\int_0^\infty T_u^V (x,y)\partial _u\Big(\frac{e^{-\frac{t^2}{4u}}}{u}\Big)du=T_u^V (x,y)\frac{e^{-\frac{t^2}{4u}}}{u}\Bigg]_{u\rightarrow 0^+}^{u\rightarrow +\infty} -\int_0^\infty \partial _u(T_u^V (x,y))\frac{e^{-\frac{t^2}{4u}}}{u}du,\quad x,y\in \mathbb R^d\mbox{ and }t>0.
$$

According to \eqref{F1} we get that
$$
0\leq T_u^V (x,y)\frac{e^{-\frac{t^2}{4u}}}{u}\leq C\frac{e^{-\frac{|x-y|^2+t^2}{4u}}}{u^{\frac{d}{2}+1}},\quad x,y\in \mathbb R^d\mbox{ and }t>0,
$$
so
$$
\lim_{u\rightarrow 0^+}T_u^V (x,y)\frac{e^{-\frac{t^2}{4u}}}{u}=\lim_{u\rightarrow +\infty}T_u^V (x,y)\frac{e^{-\frac{t^2}{4u}}}{u}=0,\quad x,y\in \mathbb R^d\mbox{ and }t>0.
$$
Therefore,
$$
\textcolor{black}{\int_0^\infty T_u^V (x,y)\partial _u\Big(\frac{e^{-\frac{t^2}{4u}}}{u}\Big)du=-\int_0^\infty \partial _u(T_u^V (x,y))\frac{e^{-\frac{t^2}{4u}}}{u}du},\quad x,y\in \mathbb R^d\mbox{ and }t>0. 
$$
Since \textcolor{black}{$\partial _u T_u^V (x,y)=-\mathcal L_{V}T_u^V (x,y)$}, $x,y\in \mathbb R^d$ and $u>0$, we obtain that
$$
\textcolor{black}{\partial ^2_ tu_f(x,t)+\frac{1}{t}\partial _tu_f(x,t)= {\frac{1}{2}}\int_{\mathbb R^d}f(y)\int_0^\infty \mathcal L_{V}(T_u^V (x,y))\frac{e^{-\frac{t^2}{4u}}}{u}dudy},\quad x\in \mathbb R^d\mbox{ and }t>0.
$$

Our next objective is to see that, for each $x\in \mathbb R^d$ and $t>0$,
\begin{equation}\label{2.2}
\int_{\mathbb R^d}f(y)\int_0^\infty \mathcal L_{V}(T_u^V (x,y))\frac{e^{-\frac{t^2}{4u}}}{u}dudy=\mathcal L_{V}\Big(\int_{\mathbb R^d}f(y)\int_0^\infty T_u^V (x,y)\frac{e^{-\frac{t^2}{4u}}}{u}dudy\Big).
\end{equation}
According to \cite[Lemmas 3.4 and 3.5]{LQWZ}, for every $N\in \mathbb N$ there exist $C,c>0$ such that, for each $x,y\in \mathbb R^d$ and $u>0$,
\begin{equation}\label{2.3}
|\nabla_xT_u^V (x,y)|\leq C\left\{
\begin{array}{ll}
\displaystyle \frac{e^{-c\frac{|x-y|^2}{u}}}{u^{\frac{d+1}{2}}}\left(1+\frac{\sqrt{u}}{\rho (x)}+\frac{\sqrt{u}}{\rho (y)}\right)^{-N},&\sqrt{u}\leq |x-y|,\\[0.5cm]
\displaystyle\frac{e^{-c\frac{|x-y|^2}{u}}}{u^{\frac{d}{2}}|x-y|}\left(1+\frac{\sqrt{u}}{\rho (x)}+\frac{\sqrt{u}}{\rho (y)}\right)^{-N},&\sqrt{u}>|x-y|,
\end{array}
\right.
\end{equation}
and
\begin{equation}\label{2.4}
|\nabla _xT_u^V (x,y)|\leq Cu^{-\frac{d+1}{2}}\left(1+\frac{\sqrt{u}}{\rho (x)}+\frac{\sqrt{u}}{\rho (y)}\right)^{-N}.
\end{equation}
Let $x_0\in \mathbb R^d$ and $t_0>0$. By using \eqref{2.3} and \eqref{2.4} we obtain that
\begin{align*}
\int_{\mathbb R^d}|f(y)|\int_0^\infty |\nabla _xT_u^V (x,y)|\frac{e^{-\frac{t^2}{4u}}}{u}dudy&\leq C\left(\int_{\textcolor{black}{|y|\leq 2|x|}}|f(y)|\int_0^\infty \frac{e^{-\frac{t^2}{4u}}}{u^{\frac{d+3}{2}}}dudy\right.\\
&\hspace{-5cm} \left.\quad +\int_{|y|> 2|x|} {\frac{|f(y)|}{|x-y|}}\int_{|x-y|^2}^\infty \frac{e^{-c\frac{|x-y|^2+t^2}{u}}}{u^{\frac{d}{2}+1}}dudy+\int_{|y|> 2|x|}|f(y)|\int_0^{|x-y|^2} \frac{e^{-c\frac{|x-y|^2+t^2}{u}}}{u^{\frac{d+3}{2}}}dudy\right)\\
&\hspace{-5cm}\leq C\left(\frac{1}{t^{d+1}}\int_{|y|\leq 2|x|}|f(y)|dy+\int_{|y|> 2|x|}|f(y)|\left(\frac{1}{|x-y|(t^2+|x-y|^2)^{\frac{d}{2}}}+\frac{1}{(t^2+|x-y|^2)^{\frac{d+1}{2}}}\right)dy\right)\\
&\hspace{-5cm}\leq C\left(\frac{1}{t^{d+1}}\int_{|y|\leq 2|x|}|f(y)|dy+\int_{|y|> 2|x|}|f(y)|\Big(\frac{1}{|y|(t+|y|)^d}+\frac{1}{(t+|y|)^{d+1}}\Big)dy\right)\\
&\hspace{-5cm}\leq C\textcolor{black}{\left(\frac{(1+|x|)^d}{t^{d+1}}+\frac{1}{|x|\min\{1,t^d\}}+\frac{1}{\min\{1,t^{d+1}\}}\right)}\int_{\mathbb R^d}\frac{|f(y)|}{(1+|y|)^d}dy\\
&\hspace{-5cm}\leq C\int_{\mathbb R^d}\frac{|f(y)|}{(1+|y|)^d}dy,\quad (x,t)\in B\Big((x_0,t_0),\frac{t_0}{2}\Big).
\end{align*}
These estimates imply that
$$
\int_{\mathbb R^d}\int_0^\infty \partial _{x_i}T_u^V (x,y)\frac{e^{-\frac{t^2}{4u}}}{u}dudy=\partial _{x_i}\int_{\mathbb R^d}\int_0^\infty T_u^V (x,y)\frac{e^{-\frac{t^2}{4u}}}{u}dudy,\quad (x,t)\in \mathbb R_+^{d+1},\, i=1,\ldots,d.
$$
According to the estimates in the bottom of \cite[p. 20]{LQWZ} there exists $C>0$ such that, when $y\in \mathbb R^d$, $x\in B(x_0,R)$ with $0<R<\rho (x_0)$ and $u>0$,
\begin{equation}\label{2.5}
|\nabla _x^2T_u^V (x,y)|\leq C\Big(R^{\frac{d}{q}-2}\sup_{z\in B(x_0,2R)}|T_u^V (z,y)|+R^{\frac{d}{q}}\sup_{z\in B(x_0,2R)}|\partial _uT_u^V (z,y)|\Big).
\end{equation}
By proceeding as in the proof of \cite[Lemma 3.7]{LQWZ} we can deduce from \eqref{2.5} that, for every $N\in \mathbb N$ there exist $C,c>0$ such that
\begin{equation}\label{2.6}
|\nabla_x^2T_u^V (x,y)|\leq C\left\{
\begin{array}{ll}
\displaystyle \frac{e^{-c\frac{|x-y|^2}{u}}}{u^{\frac{d+2-d/q}{2}}}\left(1+\frac{\sqrt{u}}{\rho (x)}+\frac{\sqrt{u}}{\rho (y)}\right)^{-N},&\sqrt{u}\leq |x-y|,\\[0.5cm]
\displaystyle\frac{e^{-c\frac{|x-y|^2}{u}}}{u^{\frac{d}{2}}|x-y|^{2-\frac{d}{q}}}\left(1+\frac{\sqrt{u}}{\rho (x)}+\frac{\sqrt{u}}{\rho (y)}\right)^{-N},&\sqrt{u}>|x-y|,
\end{array}
\right.
\end{equation}
with $x,y\in \mathbb R^d$ and $u>0$.

We are going to prove that, for every $N\in \mathbb N$ there exists $C>0$ such that
\begin{equation}\label{2.7}
|\nabla_x^2T_u^V (x,y)|\leq \frac{C}{u^{\frac{d+2-d/q}{2}}}\left(1+\frac{\sqrt{u}}{\rho (x)}+\frac{\sqrt{u}}{\rho (y)}\right)^{-N},\quad x,y\in \mathbb R^d\mbox{ and }u>0.
\end{equation}

Let $N\in \mathbb N$. By \eqref{F2} and \eqref{2.5} we obtain that, for every $x,y\in \mathbb R^d$, $u>0$ and $0<R<\rho (x)$,
\begin{align*}
|\nabla_x^2T_u^V (x,y)|&\leq  C\left(\frac{R^{\frac{d}{q}-2}}{u^{\frac{d}{2}}}+\frac{R^{\frac{d}{q}}}{u^{\frac{d}{2}+1}}\right)\left(1+\frac{\sqrt{u}}{\rho (x)}+\frac{\sqrt{u}}{\rho (y)}\right)^{-N}\\
&\leq \frac{C}{u^{\frac{d+2-d/q}{2}}}\left(1+\frac{\sqrt{u}}{\rho (x)}+\frac{\sqrt{u}}{\rho (y)}\right)^{-N}\left(\frac{R}{\sqrt{u}}\right)^{\frac{d}{q}-1}\left(\frac{\sqrt{u}}{R}+\frac{R}{\sqrt{u}}\right).
\end{align*}
\textcolor{black}{Let $u>0$ and $x\in \mathbb R^d$}. We consider the function
$$
h_u(R)=\left(\frac{R}{\sqrt{u}}\right)^{\frac{d}{q}-1}\Big(\frac{\sqrt{u}}{R}+\frac{R}{\sqrt{u}}\Big),\quad 0<R<\rho (x).
$$
Observe that this function $h_u$ is decreasing in $\left(0,\sqrt{\frac{2q-d}{d} u}\right)$. Moreover,
$$
|\nabla_x^2T_u^V (x,y)|\leq \frac{C}{u^{\frac{d+2-d/q}{2}}}\Big(1+\frac{\sqrt{u}}{\rho (x)}+\frac{\sqrt{u}}{\rho (y)}\Big)^{-N}\inf_{0<R<\rho (x)}h_u(R),\quad y\in \mathbb R^d.
$$
\textcolor{black}{Notice that
$$
\inf_{0<R<\rho (x)}h_u(R)\leq C\left(1+\frac{\sqrt{u}}{\rho (x)}\right)^2.
$$}
Indeed, if  {$\sqrt{\frac{2q-d}{d} u}\le \rho(x)$ then }
$$
\inf_{0<R<\rho (x)}h_u(R)\leq \inf_{0\leq R\leq  {\sqrt{\frac{2q-d}{d} u}}}h_u(R)=h_u\left( {\sqrt{\frac{2q-d}{d} u}}\right)= {C},
$$
while if $\rho (x)< {\sqrt{\frac{2q-d}{d} u}}$, then 
$$
\inf_{0<R<\rho (x)}h_u(R)=h_u(\rho (x))=\left(\frac{\rho (x)}{\sqrt{u}}\right)^{\frac{d}{q}}\left(1+\left(\frac{\sqrt{u}}{\rho (x)}\right)^2\right)\leq  {C}\left(1+\frac{\sqrt{u}}{\rho (x)}\right)^2.
$$
Therefore,
$$
|\nabla_x^2T_u^V (x,y)|\leq  \frac{C}{u^{\frac{d+2-d/q}{2}}}\left(1+\frac{\sqrt{u}}{\rho (x)}+\frac{\sqrt{u}}{\rho (y)}\right)^{-N}\left(1+\frac{\sqrt{u}}{\rho (x)}\right)^2,
$$
and the arbitrariness of $N$ allows us to get \eqref{2.7}. 

Now, according to \eqref{2.6} and \eqref{2.7} we have that
\begin{align*}
\int_{\mathbb R^d}|f(y)|\int_0^\infty |\nabla_x^2T_u^V (x,y)|\frac{e^{-\frac{t^2}{4u}}}{u}dudy&\leq C\left(\int_{|y|\leq 2|x|}|f(y)|\int_0^\infty \frac{e^{-\frac{t^2}{4u}}}{u^{\frac{d+2-d/q}{2}+1}}dudy\right.\\
&\hspace{-6cm}\quad +\left.\int_{|y|>2|x|}|f(y)|\left(\int_{|x-y|^2}^\infty \frac{e^{-c\frac{|x-y|^2+t^2}{u}}}{u^{\frac{d}{2}+1}|x-y|^{2-\frac{d}{q}}}du {+}\int_0^{|x-y|^2}\frac{e^{-c\frac{|x-y|^2+t^2}{u}}}{u^{\frac{d+2-d/q}{2}+1}}du\right)dy\right)\\
&\hspace{-6cm}\leq C\left(\frac{1}{t^{d+2-\frac{d}{q}}}\int_{|y|\leq 2|x|}|f(y)|dy+\int_{|y|>2|x|}\left(\frac{|f(y)|}{|x-y|^{2-\frac{d}{q}}(t+|x-y|)^d}+\frac{|f(y)|}{(t+|x-y|)^{d+2-\frac{d}{q}}}\right)dy\right)\\
&\hspace{-6cm}\leq \textcolor{black}{C\left(\frac{(1+|x|)^d}{t^{d+2-\frac{d}{q}}}+\frac{1}{|x|^{2-\frac{d}{q}}\min\{1,t^{d}\}}+\frac{1}{\min \{1,t^{d+2-\frac{d}{q}}\}}\right)}\int_{\mathbb R^d}\frac{|f(y)|}{(1+|y|)^d}dy\\
&\hspace{-6cm}\leq C\int_{\mathbb R^d}\frac{|f(y)|}{(1+|y|)^{ {d}}}dy,\quad (x,t)\in B\left((x_0,t_0),\frac{t_0}{2}\right).
\end{align*}
From the previous estimates we deduce that, for each $(x,t)\in\mathbb R^d\times (0,\infty)$ and $i=1,\ldots,d$,
$$
\int_{\mathbb R^d}f(y)\int_0^\infty \partial ^2_{x_i}T_u^V (x,y)\frac{e^{-\frac{t^2}{4u}}}{u}dudy=\partial ^2_{x_i}\int_{\mathbb R^d}f(y)\int_0^\infty T_u^V (x,y)\frac{e^{-\frac{t^2}{4u}}}{u}dudy.
$$
Thus, \eqref{2.2} is established and we conclude that
$$
\textcolor{black}{\partial ^2_ tu_f(x,t)+\frac{1}{t}\partial _tu_f(x,t) {-}\mathcal L_{V}u_f(x,t)=0},\quad x\in \mathbb R^d\mbox{ and }t>0.
$$
From the considerations above, we also conclude that $u_f\in C^1(\mathbb R^{d+1}_+)$.\\

\noindent{\bf Proof of $(iii)$}. We can write
\begin{align*}
   t\partial _tu_f(x,t) &=- {\frac{1}{2}}\int_{\mathbb R^d}f(y)\int_0^\infty T_u^V (x,y)\frac{t^2e^{-\frac{t^2}{4u}}}{\textcolor{black}{2}u^2}dudy\\
   &=- \int_{\mathbb R^d}f(y)\int_0^\infty (T_u^V (x,y)-T_u(x-y))\frac{t^2e^{-\frac{t^2}{4u}}}{\textcolor{black}{4}u^2}dudy\\
   &\quad -\int_{\mathbb R^d}f(y)\int_0^\infty T_u(x-y)\frac{t^2e^{-\frac{t^2}{4u}}}{\textcolor{black}{4}u^2}dudy,\quad x\in \mathbb R^d\mbox{ and }t>0.
\end{align*}
According to \cite[Theorem 3.1, (3.22)]{ChHW} we have that
$$
\lim_{t\rightarrow 0^+}\int_{\mathbb R^d}f(y)\int_0^\infty T_u(\cdot -y)\frac{t^2e^{-\frac{t^2}{4u}}}{4u^2}dudy=f,\quad \mbox{ in }L^1_{\rm loc}(\mathbb R^d).
$$
We are going to see that
\begin{equation}\label{*1}
\lim_{t\rightarrow 0^+}\int_{\mathbb R^d}f(y)\int_0^\infty (T_u^V (\cdot,y)-T_u(\cdot-y))\frac{t^2e^{-\frac{t^2}{4u}}}{u^2}dudy=0,\quad \mbox{ in }L^1_{\rm loc}(\mathbb R^d).    
\end{equation}
We decompose the last integral as follows:
\begin{align}\label{*2}
\int_{\mathbb R^d}f(y)\int_0^\infty (T_u^V (x,y)-T_u(x-y))\frac{t^2e^{-\frac{t^2}{4u}}}{u^2}dudy\nonumber\\
&\hspace{-4cm}=\int_{\mathbb R^d}f(y)\left(\int_0^{\rho (x)^2}+\int_{\rho (x)^2}^\infty \right)(T_u^V (x,y)-T_u(x -y))\frac{t^2e^{-\frac{t^2}{4u}}}{u^2}dudy\nonumber\\
&\hspace{-4cm}=:I_1(x,t)+I_2(x,t),\quad x\in \mathbb R^d \mbox{ and }t>0.
\end{align}

By using \eqref{F1} we get
\begin{align*}
    |I_2(x,t)|&\leq Ct^2\int_{\mathbb R^d}|f(y)|\int_{\rho (x)^2}^\infty \frac{e^{-\frac{|x-y|^2+t^2}{4u
    }}}{u^{\frac{d}{2}+2}}dudy\\
    &\leq Ct^2\left(\int_{|x-y|<\rho(x)}+\int_{\substack|x-y|\geq\rho(x)}\right)|f(y)|\int_{\rho (x)^2}^\infty \frac{e^{-\frac{|x-y|^2+t^2}{4u}}}{u^{\frac{d}{2}+2}}dudy\\
    &=:I_{2,1}(x,t)+I_{2,2}(x,t),\quad x\in \mathbb R^d\mbox{ and }t>0.
\end{align*}
We have that
\begin{align*}
I_{2,1}(x,t)&\leq Ct^2\int_{|x-y|<\rho (x)}|f(y)|\int_{\rho (x) ^2}^\infty \frac{1}{u^{\frac{d}{2}+2}}dudy {=} C\frac{t^2}{\rho (x)^{d+2}}\int_{|x-y|<\rho (x)}|f(y)|dy\\
&\leq C\frac{t^2(1+|x|+\rho (x))^d}{\rho (x)^{d+2}}\int_{\mathbb R^d}\frac{|f(y)|}{(1+|y|)^d}dy\\
&\leq C\frac{t^2(1+|x|+\rho (x))^d}{\rho (x)^{d+2}},\quad x\in \mathbb R^d\mbox{ and }t>0.
\end{align*}
On the other hand, we can write
\begin{align*}
    I_{2,2}(x,t)&\leq Ct^2\int_{|x-y|\geq\rho(x)}\frac{|f(y)|}{|x-y|^{d+2}}dy\leq C\frac{t^2}{\rho (x)^2}\int_{\mathbb R^d}\frac{|f(y)|}{(|x-y|+\rho (x))^d}dy\\
    &\leq C\frac{t^2}{\rho (x)^2}\left(\int_{|y|\leq 2|x|}+\int_{|y|>2|x|}\right)\frac{|f(y)|}{(|x-y|+\rho (x))^d}dy\\
    &\leq C\frac{t^2}{\rho (x)^2}\left(\frac{1}{\rho (x) ^d}\int_{|y|\leq 2|x|}|f(y)|dy+\int_{|y|>2|x|}\frac{|f(y)|}{(|y|+\rho (x))^d}dy\right)\\
    &\leq C\frac{t^2}{\rho (x)^2}\left(\frac{(1+|x|)^d}{\rho (x) ^d}+\int_{\mathbb R^d}\frac{|f(y)|}{(|y|+\rho (x))^d}dy\right),\quad x\in \mathbb R^d\mbox{ and }t>0.
\end{align*}

Let $\Omega$  be a compact subset of $\mathbb R^d$. Then, we can find $x_1,\ldots,x_m\in \Omega$ such that $\Omega\subset \bigcup_{j=1}^mB(x_j,\rho (x_j))$. By \cite[Lemma 2.12]{BDK}, there exists $C>1$ such that
$$
\frac{1}{C}\rho (x_j)\leq \rho (y)\leq C\rho (x_j),\quad y\in B(x_j,\rho (x_j)),\,j=1,\ldots,m.
$$
Therefore, there exist $A,B>0$ for which
\begin{equation}\label{2.10}
    A\leq \rho (y)\leq B,\quad y\in \Omega.
\end{equation}
Now, we consider $R>0$. From \eqref{2.10} and  the fact that $f\in L^1_0(\mathbb R^d)$ it follows that
$$
\int_{B(0,R)}I_2(x,t)dx\leq Ct^2,\quad t>0.
$$
Hence,
\begin{equation}\label{2.11}
    \lim_{t\rightarrow 0^+}\int_{B(0,R)}I_2(x,t)dx=0.
\end{equation}

By using \eqref{F3} we can write
$$
|I_1(x,t)|\leq Ct^2\int_{\mathbb R^d}|f(y)|\int_0^{\rho (x)^2}\Big(\frac{\sqrt{u}}{\rho (x)}\Big)^\delta |\varphi _u(x-y)|\frac{e^{-\frac{t^2}{4u}}}{u^2}dudy,\quad x\in \mathbb R^d\mbox{ and }t>0,
$$
being $\delta =2-d/q$ and $\varphi \in \mathcal S(\mathbb R^d)$ so that $\varphi _u(z)=u^{-\frac{d}{2}}\varphi (z/\sqrt{u})$, $z\in \mathbb R^d$ and $u\in (0,\infty)$.

We have that
\begin{align*}
    |I_1(x,t)|&\leq Ct^2\left(\int_{|x-y|<\rho (x)}+\int_{|x-y|\geq \rho (x)}\right)|f(y)|\int_0^{\rho (x)^2}\Big(\frac{\sqrt{u}}{\rho (x)}\Big)^\delta\frac{e^{-\frac{t^2}{4u}}}{u^2(\sqrt{u}+|x-y|)^d}dudy\\
    &=:I_{1,1}(x,t)+I_{1,2}(x,t),\quad x\in \mathbb R^d\mbox{ and }t>0.
\end{align*}
For $I_{1,2}$ we obtain
\begin{align*}
    I_{1,2}(x,t)&\leq C\frac{t^2}{\rho (x)^\delta}\int_{|x-y|\geq \rho (x)}\frac{|f(y)|}{|x-y|^d}\int_0^\infty \frac{e^{-\frac{t^2}{4u}}}{u^{2-\frac{\delta}{2}}}dudy\leq C\frac{t^\delta}{\rho (x)^\delta}\int_{\mathbb R^d}\frac{|f(y)|}{(|x-y|+\rho (x))^d}dy\\
    &\leq C\frac{t^\delta}{\rho (x)^\delta}\left(\frac{1}{\rho (x)^d}\int_{|y|\leq 2|x|}|f(y)|dy+\int_{|y|>2|x|}\frac{|f(y)|}{(|y|+\rho (x))^d}dy\right)\\
    &\leq C\frac{t^\delta}{\rho (x)^\delta}\left(\frac{(1+|x|)^d}{\rho (x)^d}\int_{\mathbb R^d}\frac{|f(y)|}{(1+|y|)^d}dy+\int_{\mathbb R^d}\frac{|f(y)|}{(|y|+\rho (x))^d}dy\right),\quad x\in \mathbb R^d\mbox{ and }t>0.
\end{align*}
Let $R>0$. By using \eqref{2.10} we get
\begin{align*}
\int_{B(0,R)}I_{1,2}(x,t)dx\leq Ct^\delta,\quad t>0,
\end{align*}
and consequently
\begin{equation}\label{2.15}
    \lim_{t\rightarrow 0^+}\int_{B(0,R)}I_{1,2}(x,t)dx=0.
\end{equation}
On the other hand, according again to \eqref{2.10} we get
\begin{align*}
   \int_{B(0,R)}I_{1,1}(x,t)dx&\leq Ct^2\int_{B(0,R)}\int_{|x-y|<\rho (x)}\int_0^{\rho (x)^2}\frac{|f(y)|e^{-\frac{t^2}{4u}}}{u^{2-\frac{\delta}{2}}(\sqrt{u}+|x-y|)^d}dudydx\\
   &\hspace{-2cm}\leq Ct^2\int_{B(0,R)}\int_{|x-y|<\rho (x)}\int_0^{\rho (x)^2}\frac{(1+|x|+\rho (x))^d|f(y)|e^{-\frac{t^2}{4u}}}{u^{2-\frac{\delta}{2}}(\sqrt{u}+|x-y|)^d(1+|y|)^d}dudydx\\
   &\hspace{-2cm}\leq  Ct^2\int_{B(0,R)}\int_{|x-y|<\rho (x)}\int_0^{\rho (x)^2}\frac{|f(y)|e^{-\frac{t^2}{4u}}}{u^{2-\frac{\delta}{2}}(\sqrt{u}+|x-y|)^d(1+|y|)^d}dudydx,\quad t>0.
\end{align*}
By \cite[Proposition 2.1]{DGMTZ}, $\rho (x)\sim \rho (y)$, provided that $|x-y|<\rho (x)$. Then, according to \eqref{2.10} there exists $c>0$ such that the set  
$$
\Omega=\big\{(x,y,u)\in \mathbb R^d\times \mathbb R^d\times (0,\infty): |x|<R,\,|x-y|<\rho (x),\,u\in (0,\rho (x)^2)\big\}
$$
is contained in 
$$
\Omega_1=\big\{(x,y,u)\in \mathbb R^d\times \mathbb R^d\times (0,\infty): |x-y|<c\rho (y),\,\frac{1}{c}\leq \rho (y)\leq c, \,u\in (0,c\rho(y)^2)\big\}.
$$
Then, we can write
\begin{align*}    
 \int_{B(0,R)}I_{1,1}(x,t)dx&\leq Ct^2\int_{\frac{1}{c}\leq \rho (y)\leq c}\int_0^{c\rho (y)^2}\int_{|x-y|<c\rho (y)^2}\frac{|f(y)|e^{-\frac{t^2}{4u}}}{u^{2-\frac{\delta}{2}}(\sqrt{u}+|x-y|)^d(1+|y|)^d}dxdudy\\
    &\leq Ct^2\int_{\frac{1}{c}\leq \rho (y)\leq c}\frac{|f(y)|}{(1+|y|)^d}\int_0^{c\rho (y)^2}\frac{e^{-\frac{t^2}{4u}}}{u^{2-\frac{\delta}{2}}}\int_0^{c\rho (y)^2}\frac{r^{d-1}}{(\sqrt{u}+r)^d}drdudy\\
    &\leq Ct^2\int_0^\infty\frac{e^{-\frac{t^2}{4u}}}{u^{2-\frac{\delta}{4}}}du\int_{\frac{1}{c}\leq \rho (y)\leq c}\frac{|f(y)|}{(1+|y|)^d}\int_0^{c\rho (y)^2}r^{\frac{\delta}{2
    }-1}drdy\\
    &\leq Ct^{\frac{\delta}{2}}\int_{\frac{1}{c}\leq \rho (y)\leq c}\frac{|f(y)|\rho (y)^\delta}{(1+|y|)^d}dy\leq Ct^{\frac{\delta}{2}}\int_{\mathbb R^d}\frac{|f(y)|}{(1+|y|)^d}dy,\quad t>0.
\end{align*}
It follows that
\begin{equation}\label{2.16}
    \lim_{t\rightarrow 0^+}\int_{B(0,R)}I_{1,1}(x,t)dx=0.
\end{equation}
By combining \eqref{2.15} and \eqref{2.16} we obtain that, 
\begin{equation}\label{2.17}
    \lim_{t\rightarrow 0^+}\int_{B(0,R)}I_1(x,t)dx=0.
\end{equation}
By putting together \eqref{*2}, \eqref{2.11} and \eqref{2.17} we conclude that \eqref{*1} holds. Thus, we have proved that 
$$
\lim_{t\rightarrow 0^+}t\partial _tu_f(\cdot ,t)=\textcolor{black}{-f},\quad \mbox{ in }L^1_{\rm loc}(\mathbb R^d).
$$
\\
\noindent{\bf Proof of $(iv)$}. According to \cite[Lemma 3.6]{ChHW}, since  $u_f\in \mathcal {AG}(\mathbb R^{d+1}_+)\cap C^1(\mathbb R^{d+1}_+)$, it follows that 
$$
\lim_{t\rightarrow 0^+}\frac{u_f(\cdot ,t)}{\log t}=-f,\quad \mbox{ in }L^1_{\rm loc}(\mathbb R^d).
$$
\\
\noindent{\bf Proof of $(v)$}. Suppose that $f$ is a Dini continuous function at the point $x\in \mathbb R^d$. We are going to see that
\begin{equation}\label{logextDini}
(\log \mathcal L_V)(f)(x)=-2\lim_{t\rightarrow 0^+}\big(u_f(x,t)+f(x)\log t\big)-f(x)h(x).
\end{equation}
We decompose $u_f$ in the following way:
\begin{align}\label{decomp}
    2u_f(x,t)&=\left(\int_{\mathbb R^d\setminus B(x,1)}+\int_{B(x,1)\setminus B(x,t)}+\int_{B(x,t)}\right)f(y)\int_0^\infty T_u^V (x,y)\frac{e^{-\frac{t^2}{4u}}}{u}dudy\nonumber\\
    &=:J_1(x,t)+J_2(x,t)+J_3(x,t),\quad t\in(0,1),
\end{align}

First we are going to show that 
\begin{equation}\label{J1}
   \lim_{t\rightarrow 0^+}J_1(x,t)= \int_{\mathbb R^d\setminus B(x,1)}f(y)\int_0^\infty \frac{T_u^V (x,y)}{u}dudy.
\end{equation}
By using \eqref{F1} we get
\begin{align*}
\left|J_1(x,t)-\int_{\mathbb R^d\setminus B(x,1)}f(y)\int_0^\infty \frac{T_u^V (x,y)}{u}dudy\right|&\leq \int_{\mathbb R^d\setminus B(x,1)}|f(y)|\int_0^\infty T_u^V (x,y)\Big|\frac{e^{-\frac{t^2}{4u}}-1}{u}\Big|dudy\\
&\hspace{-6cm}\leq Ct^2\int_{\mathbb R^d\setminus B(x,1)}|f(y)|\int_0^\infty \frac{e^{-\frac{t^2+|x-y|^2}{4u}}}{u^{\frac{d}{2}+2}}dudy\leq Ct^2\int_{\mathbb R^d\setminus B(x,1)}\frac{|f(y)|}{|x-y|^{d+2}}dy\\
&\hspace{-6cm}\leq Ct^2\int_{\mathbb R^d}\frac{|f(y)|}{(1+|x-y|)^{d+2}}dy\leq Ct^2\left(\int_{|y|\leq 2|x|}+\int_{|y|>2|x|}\right)\frac{|f(y)|}{(1+|x-y|)^d}dy\\
&\hspace{-6cm}\leq Ct^2\left((1+|x|)^d\int_{\mathbb R^d}\frac{|f(y)|}{(1+|y|)^d}dy+\int_{\mathbb R^d}\frac{|f(y)|}{(1+|y|)^d}dy\right)\leq Ct^2(1+|x|)^d,\quad t\in (0,1).
\end{align*}
Thus, \eqref{J1} is proved.

Our next objective is to see that
\begin{equation}\label{J3}
 \lim_{t\rightarrow 0^+}\left(J_3(x,t)-\int_{B(x,t)}f(y)\int_0^\infty T_u(x-y)\frac{e^{-\frac{t^2}{4u}}}{u}dudy\right)=0,
\end{equation}
and then, in virtue of \cite[(4.50)]{ChHW},  deduce that
\begin{equation}\label{J3alpha}
\lim_{t\rightarrow 0^+}J_3(x,t)=\alpha _df(x),
\end{equation}
where $\alpha _d=2\int_0^1(1+t)^{-\frac{d}{2}}t^{\frac{d}{2}-1}dt$.

\noindent From  \eqref{F3} we have that
\begin{align*}
D(x,t)&:=\left|J_3(x,t)-\int_{B(x,t)}f(y)\int_0^\infty T_u(x-y)\frac{e^{-\frac{t^2}{4u}}}{u}dudy\right|\\
&\leq \int_{B(x,t)}|f(y)|\int_0^\infty \frac{e^{-\frac{t^2}{4u}}}{u}|T_u^V (x,y)-T_u(x-y)|dudy\\
&\leq \int_{B(x,t)}|f(y)|\left(\int_0^{\rho (x)^2}\frac{e^{-\frac{t^2}{4u}}}{u}\Big(\frac{\sqrt{u}}{\rho (x)}\Big)^\delta|\varphi_u(x-y)|du+\int_{\rho (x)^2}^\infty\frac{e^{-\frac{t^2}{4u}}}{u}|\varphi_u(x-y)|du\right)dy\\
&\leq C\int_{B(x,t)}|f(y)|\left(\frac{1}{\rho (x)^\delta}\int_0^{\rho (x)^2}\frac{e^{-\frac{t^2}{4u}}}{u^{1-\frac{\delta}{2}}(\sqrt{u}+|x-y|)^d}du+\int_{\rho (x)^2}^\infty\frac{e^{-\frac{t^2}{4u}}}{u(\sqrt{u}+|x-y|)^d}du\right)dy\\
&\leq C\int_{B(x,t)}|f(y)|dy\left(\frac{1}{\rho (x)^\delta}\int_0^{\rho (x)^2}\frac{e^{-\frac{t^2}{4u}}}{u^{\frac{d-\delta}{2}+1}}du+\int_{\rho (x)^2}^\infty \frac{1}{u^{\frac{d}{2}+1}}du \right)\\
&\leq C\int_{B(x,t)}|f(y)|dy\left(\frac{1}{\rho (x)^\delta t^{d-\delta}}+\frac{1}{\rho (x)^d}\right),\quad t\in (0,1).
\end{align*}
By taking into account that $f$ is Dini continuous at $x$ we get
\begin{align*}
   \int_{B(x,t)}|f(y)|dy&\leq C\left(\int_{B(x,t)}|f(y)-f(x)|dy+t^d|f(x)|\right)\\
   &\hspace{-1.5cm}\leq C\left(\int_{B(x,t)}\sup_{\textcolor{black}{|z-x|\leq |y-x|}}|f(z)-f(x)|dy+t^d|f(x)|\right)\leq C\left(\int_0^tw_{f,x}(r)r^{d-1}dr+t^d|f(x)|\right)\\
    &\hspace{-1.5cm}\leq Ct^d\left(\int_0^1\frac{w_{f,x}(r)}
    {r}dr+|f(x)|\right)\leq Ct^d(1+|f(x)|),\quad t\in (0,1).
\end{align*}
Then, it follows that
$$
D(x,t)\leq  {C\left(\frac{t^\delta}{\rho (x)^\delta}+\frac{t^d}{\rho (x)^d}\right)(1+|f(x)|), \quad t\in (0,1).} 
$$
Since $\delta >0$, we conclude that $D(x,t)\longrightarrow 0$, as $t\rightarrow 0^+$, and \eqref{J3} is established.

Now we consider
\begin{align}\label{decompJ2}
    J_2(x,t)&=\int_{B(x,1)\setminus B(x,t)}(f(y)-f(x))\int_0^\infty T_u^V (x,y)\frac{e^{-\frac{t^2}{4u}}}{u}dudy\nonumber\\
    &\quad +f(x)\int_{B(x,1)\setminus B(x,t)}\int_0^\infty T_u^V (x,y)\frac{e^{-\frac{t^2}{4u}}}{u}dudy\nonumber\\
    &=:J_{2,1}(x,t)+J_{2,2}(x,t),\quad t\in(0,1).
\end{align}
First, we shall see that
\begin{equation}\label{J21}
 \lim_{t\rightarrow 0^+}J_{2,1}(x,t)= \int_{B(x,1)}(f(y)-f(x))\int_0^\infty \frac{T_u^V(x,y)}{u}dudy.
\end{equation}
We write
\begin{align*}
    F(x,t):=&\left|J_{2,1}(x,t)-\int_{ B(x,1)}(f(y)-f(x))\int_0^\infty \frac{T_u^V(x,y)}{u}dudy\right|&\\
    \leq &\left|\int_{B(x,1)\setminus B(x,t)}(f(y)-f(x))\int_0^\infty T_u^V (x,y)\frac{e^{-\frac{t^2}{4u}}-1}{u}dudy\right|\\
    &\quad +\left|\int_{B(x,t)}(f(y)-f(x))\int_0^\infty \frac{T_u^V (x,y)}{u}dudy\right|\\
    =:&J_{2,1,1}(x,t)+J_{2,1,2}(x,t),\quad t\in (0,1).
\end{align*}
By using \eqref{F1} and taking into account that $f$ is a Dini continuous function at $x$ we obtain
\begin{align*}
  J_{2,1,2}(x,t)&\leq C\int_{B(x,t)}|f(y)-f(x)|\int_0^\infty \frac{e^{-\frac{|x-y|^2}{4u}}}{u^{\frac{d}{2}+1}}dudy\\
  &\leq C\int_{B(x,t)}\frac{|f(y)-f(x)|}{|x-y|^d}dy\leq C\int_{B(x,t)}\frac{\sup_{|z-x|\leq |y-x|}|f(z)-f(x)|}{|x-y|^d}dy\\
  &\leq C\int_0^t\frac{w_{f,x}(r)}{r}dr,\quad t\in (0,1).
\end{align*}
Hence, $J_{2,1,2}(x,t)\longrightarrow 0$, as $t\rightarrow 0^+$. Similarly, estimate \eqref{F1} allows us to get
\begin{align*}
    J_{2,1,1}(x,t)&\leq C\int_{B(x,1)\setminus B(x,t)}|f(y)-f(x)|\int_0^\infty \frac{t^2e^{-\frac{t^2+|x-y|^2}{4u}}}{u^{\frac{d}{2}+2}}dudy\\
    &\leq C\int_{B(x,1)}\frac{t^2|f(y)-f(x)|}{(t+|x-y|)^{d+2}}dy\leq C\int_{B(x,1)}\frac{|f(y)-f(x)|}{|x-y|^d}dy\\
    &\leq C\int_0^1\frac{w_{f,x}(r)}{r}dr<\infty,\quad t\in (0,1).
\end{align*}
Then, applying the dominated convergence theorem we obtain that $J_{2,1,1}(x,t)\longrightarrow 0$, as $t\rightarrow 0^+$. Thus, we conclude that
$$
\lim_{t\rightarrow 0^+}F(x,t)=0, 
$$
and \eqref{J21} is proved.

Finally, let us show that
\begin{equation}\label{J22}
    \lim_{t\rightarrow 0^+}(J_{2,2}(x,t)+2f(x)\log t)=f(x)g(x),
\end{equation}
where 
$$
g(x)=\beta_d+\int_{B(x,1)}\int_0^\infty \frac{T_u^V(x,y)-T_u(x-y)}{u}dudy,
$$
  being $\beta_d=\displaystyle 2 \int_1^\infty (r^2+1)^{-d/2}-r^{-d})r^{d-1}dr$ and the double integral is absolutely convergent.
 
We write
\begin{align*}
    J_{2,2}(x,t)+2f(x)\log t&=f(x)\Bigg(\int_{B(x,1)\setminus B(x,t)}\int_0^\infty (T_u^V (x,y)-T_u(x-y))\frac{e^{-\frac{t^2}{4u}}}{u}dudy\\
    &\quad +\left(\int_{B(x,1)\setminus B(x,t)}\int_0^\infty T_u(x-y)\frac{e^{-\frac{t^2}{4u}}}{u}dudy+2\log t\right)\Bigg)\\
    &=:f(x)(J_{2,2,1}(x,t)+J_{2,2,2}(x,t)),\quad t\in (0,1).
\end{align*}
According to \cite[Lemma 4.1]{ChHW} we get
$$
\lim_{t\rightarrow 0^+}J_{2,2,2}(x,t)=\beta_d.
$$
Thus, to establish \eqref{J22} we only need to show that
$$
\lim_{t\rightarrow 0^+}J_{2,2,1}(x,t)=\int_{B(x,1)}\int_0^\infty \frac{T_u^V (x,y)-T_u(x-y)}{u}dudy.
$$
We decompose $J_{2,2,1}$ as follows:
\begin{align*}
J_{2,2,1}(x,t)&=\int_{B(x,1)\setminus B(x,t)}\left(\int_0^{\rho (x)^2}+\int_{\rho (x)^2}^\infty\right) (T_u^V (x,y)-T_u(x-y))\frac{e^{-\frac{t^2}{4u}}}{u}dudy\\
&=:H_1(x,t)+H_2(x,t),\quad t\in (0,1).
\end{align*}
According to \eqref{F3} we can write
\begin{align*}
     |H_1(x,t)|&\leq C\int_{B(x,1)\setminus B(x,t)}\int_0^{\rho (x)^2}\left(\frac{\sqrt{u}}{\rho (x)}\right)^\delta |\varphi _u(x-y)|\frac{e^{-\frac{t^2}{4u}}}{ {u}}dudy\\
     &\leq \frac{C}{\rho (x) ^\delta}\int_{B(x,1)}\int_0^{\rho (x)^2}\frac{u^{\frac{\delta}{2}}}{(\sqrt{u}+|x-y|)^d}\frac{e^{-\frac{t^2}{4u}}}{ {u}}dudy\\
     &\leq \frac{C}{\rho (x) ^\delta}\int_0^1\int_0^{\rho (x)^2}\frac{u^{\frac{\delta}{2} {-1}}r^{d-1}}{(\sqrt{u}+r)^d}dudr\leq  {\frac{C}{\rho (x) ^\delta}  \int_0^1\int_0^{\rho (x)^2}\frac{u^{\frac{\delta}{4}{-1}}r^{d-1}}{(\sqrt{u}+r)^{d-\delta/2}}dudr}\\
      & {=\frac{C}{\rho (x) ^\delta}  \int_0^1 r^{\delta/2-1}\int_0^{\rho (x)^2}\frac{u^{\frac{\delta}{4}{-1}}}{\left(\frac{\sqrt{u}}{r}+1\right)^{d-\delta/2}}dudr}\\
     & {\leq \frac{C}{\rho (x) ^\delta}\int_0^{\rho (x)^2}u^{\frac{\delta}{4}{-1}}du \le \frac{C}{\rho (x) ^{\delta/2}} }
     ,\quad t\in (0,1).
\end{align*}
By applying the dominated convergence theorem we obtain
$$
\lim_{t\rightarrow 0^+}H_1(x,t)=\int_{B(x,1)}\int_0^{\rho (x)^2} \frac{T_u^V (x,y)-T_u (x-y)}{u}dudy,
$$
being the double integral absolutely convergent.

On the other hand, by \eqref{F1} it follows that
$$
|H_2(x,t)|\leq C\int_{B(x,1)}\int_{\rho (x)^2}^\infty \frac{e^{-\frac{t^2+|x-y|^2}{4u}}}{u^{\frac{d}{2}+1}}dudy\leq C\int_{\rho (x)^2}^\infty \frac{du}{u^{\frac{d}{2}+1}}\leq \frac{C}{\rho (x)^d},\quad t\in (0,1),
$$
and by using again the dominated convergence theorem we get
$$
\lim_{t\rightarrow 0^+}H_2(x,t)=\int_{B(x,1)}\int_{\rho (x)^2}^\infty \frac{T_u^V (x,y)-T_u(x-y)}{u}dudy,
$$
being the double integral absolutely convergent.

The above estimations lead to \eqref{J22}.

By putting together \eqref{decomp}, \eqref{J1}, \eqref{J3alpha}, \eqref{decompJ2}, \eqref{J21},  and \eqref{J22} and according to \eqref{eq1.6} we conclude that 
\begin{align*}
    -2\lim_{t\rightarrow 0^+}\big(u_f(x,t)+f(x)\log t\big)&\\
    &\hspace{-3cm}=-\int_{\mathbb R^d\setminus B(x,1)}f(y)\int_0^\infty \frac{T_u^V (x,y)}{u}dudy-\int_{B(x,1)}\!\!\!(f(y)-f(x))\int_0^\infty \frac{T_u^V (x,y)}{u}dudy\\
    &\hspace{-3cm}\quad -f(x)\left(\int_{B(x,1)}\int_0^\infty\frac{T_u^V (x,y)-T_u(x-y)}{u}dudy+\alpha_d+\beta_d\right)\\
    &\hspace{-3cm}=((\log \mathcal L_V)f)(x)+f(x)\Big(K(x)-\int_{B(x,1)}\int_0^\infty\frac{T_u^V (x,y)-T_u(x-y)}{u}dudy-\alpha_d-\beta_d\Big)\\
    &\hspace{-3cm}=((\log \mathcal L_V)f)(x)+f(x)h(x).
\end{align*}

We now establish the representation \eqref{logextDini} in a distributional sense, assuming only that $f\in L^1_0(\mathbb R^d)$. We have to show that, for every $\varphi \in C_c^\infty (\mathbb R^d)$,
\begin{align}\label{2.26}
    \int_{\mathbb R^d}f(x)((\log \mathcal L_V)\varphi )(x)dx&=-2\lim_{t\rightarrow 0^+}\int_{\mathbb R^d}(u_f(x,t)+f(x)\log t)\varphi (x)dx\nonumber\\
    &\quad -\int_{\mathbb R^d}f(x)h(x)\varphi (x)dx.
\end{align}
Actually, as in \cite[Proposition 4.4]{ChHW}, we can prove \eqref{2.26} for $\varphi \in C_{c,D}(\mathbb R^d)$, that is, $\varphi$ is uniformly Dini continuous on $\mathbb R^d$ with compact support.

Let $\varphi \in C_{c,D}(\mathbb R^d)$ and $\Omega={\rm supp}\,\varphi$. There exists a collection $\{z_j\}_{j=1}^n\subseteq \Omega$ such that
$\Omega\subset \bigcup_{j=1}^nB(z_j,1)$. In virtue of \eqref{bound} we can write
\begin{align*}
    \int_{\mathbb R^d}|u_f(x,t)||\varphi (x)|dx&\leq \frac{C}{\min\{1,t^d\}}\int_{\Omega}|\varphi (x)|(1+|x|)^ddx\\
    &\leq \frac{C}{\min\{1,t^d\}}\sum_{j=1}^n\int_{B(z_j,1)}(|\varphi (x)-\varphi (z_j)|+|\varphi (z_j)|)dx\\
    &\leq \frac{C}{\min\{1,t^d\}}\left(\sum_{j=1}^n\int_{B(z_j,1)}\sup_{|z-z_j|\leq |x-z_j|}|\varphi (z)-\varphi (z_j)|dx+1\right)\\
    &\leq \frac{C}{\min\{1,t^d\}}\left(\sum_{j=1}^n\int_0^1\frac{w_{f,z_j}(r)}{r}r^ddr+1\right)<\infty,\quad t>0.
\end{align*}
Then, Fubini's theorem  {and the symmetry of the kernel} imply that 
\begin{equation}\label{2.27}
\int_{\mathbb R^d}u_f(x,t)\varphi (x)dx=\int_{\mathbb R^d}f(x)u_\varphi (x,t)dx,\quad t>0.
\end{equation}
On the other hand, we have that
\begin{equation}\label{2.28}
\int_{\mathbb R^d}f(x)\lim_{t\rightarrow 0^+}(u_\varphi(x,t)+\varphi (x)\log t)dx=\lim_{t\rightarrow 0^+}\int_{\mathbb R^d}f(x)(u_\varphi (x,t)+\varphi (x)\log t)dx.
\end{equation}
Indeed, since $f\in L^1_0(\mathbb R^d),$ if we prove that
\begin{equation}\label{2.29}
    |u_\varphi (x,t)+\varphi (x)\log t|\leq \frac{C}{(1+|x|)^d},\quad x\in \mathbb R^d\mbox{ and }t\in (0,1),
\end{equation}
then dominated convergence theorem leads to the identity in \eqref{2.28}.

Let $R>1$ such that $\Omega \subset B(0,R)$. By using \eqref{F1} we get
$$
|u_\varphi(x,t)+\varphi (x)\log t| {=|u_\varphi(x,t)|}\leq C\int_\Omega\frac{|\varphi (y)|}{|x-y|^d}dy\leq \frac{C}{|x|^d}\leq \frac{C}{(1+|x|)^d},\quad |x|\geq 2R\mbox{ and }t>0.
$$
On the other hand, since $\varphi\in C_{c,D}(\mathbb R^d)$ a careful reading of the proof of the pointwise representation \eqref{logextDini} allows us to conclude that there exists a constant $C>0$ such that
$$
|u_\varphi (x,t)+\varphi (x)\log t|\leq C\leq \frac{C}{(1+|x|)^d},\quad x\in B(0,R)\mbox{ and }t\in (0,1).
$$
Thus, \eqref{2.29} is proved.

By putting together \eqref{2.27} and \eqref{2.28} we obtain \eqref{2.26}.

\edproof
\section{The discrete setting}\label{sec3} 
The heat semigroup associated with $\Delta_d$ is given by the convolution 
$$
e^{u\Delta_d}f(k):=\sum_{j\in\Z}p_u(k-j)f(j),
$$
where $p_u(k)=e^{-2u}I_{|k|}(2u)$, $k\in\Z$ and $u>0$, being $I_\nu$ the modified Bessel function of the first kind and order $k\in\N\cup\{0\}.$
 According to  {\cite[(5.10.22)]{Leb}} we can write
\begin{align}\label{3.2}
    0\leq p_u(k)&=\frac{ {u}^{|k|}e^{-2u}}{ {\sqrt{\pi}}\Gamma (|k|+\frac{1}{2})}\int_{-1}^1e^{-2us}(1-s^2)^{|k|-\frac{1}{2}}ds\nonumber\\
    &\leq \frac{2u^{|k|}}{\sqrt{\pi}\Gamma (|k|+\frac{1}{2})}\int_0^1(1-s^2)^{|k|-\frac{1}{2}}ds=\frac{u^{|k|}}{ {\sqrt{\pi}}\Gamma (|k|+\frac{1}{2})}\int_0^1(1-z)^{|k|-\frac{1}{2}}z^{-\frac{1}{2}}dz\nonumber\\
    &=\frac{ {u^{|k|}}}{\Gamma (|k|+1)}\leq \black{\frac{2u^{|k|}}{\sqrt{1+|k|}}},\quad k\in \mathbb Z\mbox{ and }u>0.
\end{align}
So
\begin{equation}\label{3.3}
0\leq p_u(k)\leq \black{\frac{2}{\sqrt{1+|k|}}},\quad k\in \mathbb Z\mbox{ and }u\in (0,1].
\end{equation}
On the other hand, according to \cite[Lemma 2.1]{AL-CM} it holds that
\begin{equation}\label{3.4}
    0\leq p_u(k)\leq C\frac{u^{-\frac{1}{4}}}{\sqrt{1+|k|}},\quad k\in \mathbb Z\mbox{ and }u\geq 1.
\end{equation}
\subsection{Proof of Theorem \ref{Th1.2}}

Let $f=(f(m))_{m\in \mathbb Z}$ such that $\sum_{m\in \Z}|f(m)|(1+|m|)^{-1/2}<\infty$. \textcolor{black}{ Firstly, let us show that}
\begin{equation}\label{3.1}
u_f(n,t)=\sum_{m\in \mathbb Z}f(m)\int_0^\infty p_u(n-m)\frac{e^{-\frac{t^2}{4u}}}{u}du,\quad n\in \mathbb Z \mbox{ and }t>0.
\end{equation} 
For that it is sufficient to check that
$$
\sum_{m\in \mathbb Z}|f(m)|\int_0^\infty p_u(n-m)\frac{e^{-\frac{t^2}{4u}}}{u}du<\infty,\quad n\in \mathbb Z\mbox{ and }t>0.
$$
By using \eqref{3.3} \black{and \eqref{3.4}} we can write
\begin{align}\label{acotacion}
    \sum_{m\in \mathbb Z}|f(m)|\black{\int_0^{\infty}}p_u(n-m)\frac{e^{-\frac{t^2}{4u}}}{u}du&\leq C\sum_{m\in \mathbb Z}\frac{|f(m)|}{\sqrt{1+|n-m|}}\Big(\int_0^1\frac{e^{-\frac{t^2}{4u}}}{u}du+\int_1^\infty\frac{e^{-\frac{t^2}{4u}}}{u^{\frac{5}{4}}}du\Big)\nonumber\\
    &\hspace{-3cm}\leq C\sum_{m\in \mathbb Z}\frac{|f(m)|}{\sqrt{1+|n-m|}}\Big(\frac{1}{t^2}\int_0^1du+\frac{1}{t^{\frac{1}{4}}}\int_1^\infty\frac{du}{u^{\frac{9}{8}}}\Big)\nonumber\\
    &\hspace{-3cm}\leq C\Big(\frac{1}{t^2}+\frac{1}{t^{\frac{1}{4}}}\Big)\sqrt{1+|n|}\sum_{m\in \mathbb Z}\frac{|f(m)|}{\sqrt{1+|m|}}<\infty ,\quad n\in \Z\mbox{ and }t>0.
\end{align}
Thus, \eqref{3.1} is established. Furthermore, according to \eqref{acotacion} it follows that
$$
|u_f(n,t)|\leq C\Big(\frac{1}{t^2}+\frac{1}{t^{\frac{1}{4}}}\Big)\sqrt{1+|n|},\quad n\in \mathbb Z\mbox{ and }t>0,
$$
which leads to estimation in $(i)$.

Let us now prove property $(ii)$. Considering \eqref{3.1} we can see that
\begin{equation}\label{3.8}
\partial _tu_f(n,t)=-\sum_{m\in \mathbb Z}f(m)\int_0^\infty p_u(n-m)\frac{t}{2u^2}e^{-\frac{t^2}{4u}}du,\quad n\in \mathbb Z\mbox{ and }t>0,    
\end{equation}
 and 
 
 \begin{equation}\label{3.9}
\partial _t^2u_f(n,t)= {\sum_{m\in \mathbb Z}}f(m)\int_0^\infty p_u(n-m)\Big(\frac{t^2}{4u^3}-\frac{1}{2u^2}\Big)e^{-\frac{t^2}{4u}}du,\quad n\in \mathbb Z\mbox{ and }t>0.    
\end{equation}
Indeed, let $t_0>0$ and $n_0\in \mathbb Z$. \black{By taking into account that
$$
\left(\frac{t}{u^2}+\frac{t^2}{u^3}+\frac{1}{u^2}\right)e^{-\frac{t^2}{4u}}\leq C\left(\frac{1}{t}+\frac{1}{t^2}\right)\frac{e^{-c\frac{t^2}{u}}}{u},\quad t, u\in (0,\infty),
$$
and proceeding as in \eqref{acotacion} we get that
\begin{align*}
    \sum_{m\in \mathbb Z}|f(m)|\int_0^\infty \!\!\!p_u(n_0-m)\left(\frac{t}{2u^2}+\frac{t^2}{4u^3}+\frac{1}{2u^2}\right)e^{-\frac{t^2}{4u}}du&\leq C(n_0,t_0)\sum_{m\in \mathbb Z}\frac{|f(m)|}{\sqrt{1+|m|}},\quad \frac{t_0}{2}<t<t_0,
\end{align*}
for certain $C(n_0,t_0)>0$.} Thus,  {differentiations} inside the sums and the integrals in \eqref{3.8} and \eqref{3.9}  {are justified}. 

By combining \eqref{3.8} and \eqref{3.9} we obtain
\begin{align*}
 \partial _t^2u_f(n,t)+\frac{1}{t}\partial _tu_f(n,t)&=\sum_{m\in \mathbb Z}f(m)\int_0^\infty p_u(n-m) \Big(\frac{t^2}{4u^3}-\frac{1}{ {u^2}}\Big)e^{-\frac{t^2}{4u}}du\\
 &=\sum_{m\in \mathbb Z}f(m)\int_0^\infty p_u(n-m)\partial _u\Big(\frac{e^{-\frac{t^2}{4u}}}{u}\Big)du,\quad n\in \mathbb Z\mbox{ and }t>0.
\end{align*}
Integration by parts allows us to write, for each $n\in \mathbb Z$ and $t>0$,
$$
\int_0^\infty p_u(n-m)\partial _u\Big(\frac{e^{-\frac{t^2}{4u}}}{u}\Big)du=p_u(n-m)\frac{e^{-\frac{t^2}{4u}}}{u}\Bigg]_{u\rightarrow 0}^{u\rightarrow +\infty}-\int_0^\infty \partial _u(p_u(n-m))\frac{e^{-\frac{t^2}{4u}}}{u}du.
$$
Since $I_\nu (u)\sim e^u/\sqrt{\black{2}\pi u}$, as $u\rightarrow +\infty$, for every $\nu >-1/2$, it follows that
$$
\int_0^\infty p_u(n-m)\partial _u\Big(\frac{e^{-\frac{t^2}{4u}}}{u}\Big)du=-\int_0^\infty \partial _u(p_u(n-m))\frac{e^{-\frac{t^2}{4u}}}{u}du,\quad n\in \mathbb Z\mbox{ and }t>0.
$$
Therefore, we obtain that
\begin{align*}
    \partial _t^2u_f(n,t)+\frac{1}{t}\partial _tu_f(n,t)&=-\sum_{m\in \mathbb Z}f(m)\int_0^\infty \partial _u(p_u(n-m))\frac{e^{-\frac{t^2}{4u}}}{u}du\\
    &=-\sum_{m\in \mathbb Z}f(m)\int_0^\infty (\Delta_d p_u)(n-m)\frac{e^{-\frac{t^2}{4u}}}{u}du\\
    &=-\Delta_du_f(n,t),\quad n\in \mathbb Z\mbox{ and }t>0,
\end{align*}
and $(ii)$ is established.

\vspace*{0.2cm}

Next we show $(iii)$. From \eqref{3.8} we can write
\begin{align*}
t\partial_tu_f(n,t)&=-\sum_{\substack{m\in \mathbb Z\\m\not=n}}f(m)\int_0^\infty p_u(n-m)\frac{t^2}{2u^2}e^{-\frac{t^2}{4u}}du-f(n)\int_0^\infty p_u(0)\frac{t^2}{2u^2}e^{-\frac{t^2}{4u}}du\\
&=:J_1(n,t)+J_2(n,t),\quad n\in \mathbb Z\mbox{ and }t>0.   
\end{align*}
By using \eqref{3.2} and \eqref{3.4} we get that
\begin{align*}
    |J_1( {n},t)|&\leq C\sum_{\substack{m\in \mathbb Z\\m\not=n}}\frac{|f(m)|}{1+\sqrt{|n-m|}}\left(\int_0^1 \frac{t^2}{u}e^{-\frac{t^2}{4u}}du+\int_1^\infty \frac{t^2}{u^{9/4}}e^{-\frac{t^2}{4u}}du\right)\\
    &\leq Ct\sqrt{1+|n|}\sum_{m\in \mathbb Z}\frac{|f(m)|}{\sqrt{1+|m|}}\left(\int_0^1\frac{du}{\sqrt{u}}+\int_1^\infty \frac{du}{u^{ {7/4}}}\right)\leq Ct\sqrt{1+|n|},\quad n\in \mathbb Z\mbox{ and }t\in(0,1).
\end{align*}
Then,
$$
\lim_{t\rightarrow 0^+}J_1(n,t)=0,\quad n\in \mathbb Z.
$$
On the other hand, we have that
$$
\lim_{t\rightarrow 0^+}J_2(n,t)=-2f(n),\quad n\in \mathbb Z. 
$$
Indeed, by a change of variables we can write
$$
\int_0^\infty p_u(0)\frac{t^2}{2u^2}e^{-\frac{t^2}{4u}}du=2\int_0^\infty p_{\frac{t^2}{4v}}(0)e^{-v}dv,\quad t>0,
$$
so by taking into account that $|p_z(0)|\leq 1$, $z>0$, and that $\lim_{z\rightarrow 0^+}p_z(0)=1$, the dominated convergence theorem leads to
$$
\lim_{t\rightarrow 0^+}
\int_0^\infty p_u(0)\frac{t^2}{2u^2}e^{-\frac{t^2}{4u}}du=2\int_0^\infty e^{-v}dv=2.
$$
By combining the above results $(iii)$ is obtained.

\vspace*{0.2cm}

Next, we shall see the relation between $u_f$ and $\log(-\Delta_d)f$, that is, $(v)$ in Theorem \ref{Th1.2}. By considering \eqref{3.1} we decompose $u_f$ as follows:
\begin{align}\label{I1I2I3}
    u_f(n,t)&=\sum_{m\in \mathbb Z}f(m)\int_1^\infty p_u(n-m)\frac{e^{-\frac{t^2}{4u}}}{u}du+\sum_{\substack{m\in \mathbb Z\\m\neq n}}(f(m)-f(n))\int_0^1 p_u(n-m)\frac{e^{-\frac{t^2}{4u}}}{u}du\nonumber\\
    &\quad +f(n)\sum_{m\in \mathbb Z}\int_0^1 p_u(n-m)\frac{e^{-\frac{t^2}{4u}}}{u}du\nonumber\\
    &=:S_1(n,t)+S_2(n,t)+S_3(n,t),\quad n\in \mathbb Z\mbox{ and }t>0.
\end{align}
By using \eqref{3.4} we deduce that
\begin{align*}
    \left|S_1(n,t)-\sum_{m\in \mathbb Z}f(m)\int_1^\infty \frac{p_u(n-m)}{u}du\right|&\leq \sum_{m\in \mathbb Z}|f(m)|\int_1^\infty p_u(n-m)\frac{|e^{-\frac{t^2}{4u}}-1|}{u}du\\
    &\hspace{-3cm}\leq C\sum_{m\in \mathbb Z}|f(m)|\int_1^\infty \frac{u^{-\frac{1}{4}}}{\sqrt{1+|n-m|}}\frac{t^2}{u^2}du\leq Ct^2\sum_{m\in \mathbb Z}\frac{|f(m)|}{\sqrt{1+|n-m|}}\\
    &\hspace{-3cm}\leq Ct^2\sqrt{1+|n|}\sum_{m\in \mathbb Z}\frac{|f(m)|}{\sqrt{1+|m|}},\quad n\in \mathbb Z\mbox{ and }\black{t\in (0,1)}.
\end{align*}
Then,
\begin{equation}\label{3.5}
\lim_{t\rightarrow 0^+}S_1(n,t)=\sum_{m\in \mathbb Z}f(m)\int_1^\infty \frac{p_u(n-m)}{u}du,\quad n\in \mathbb Z.
\end{equation}

On the other hand, according to \eqref{3.2}, for every $n\in \mathbb Z$ and $t\in(0,1)$, we get
\begin{align*}
    \left|S_2(n,t)-\sum_{\substack{m\in \mathbb Z\\m\neq n}}(f(m)-f(n))\int_0^1\frac{p_u(n-m)}{u}du\right|&\\
    &\hspace{-2cm}\leq C\sum_{\substack{m\in \mathbb Z\\m\neq n}}|f(m)-f(n)|\int_0^1\frac{u^{|n-m|}}{\Gamma (|n-m|+1)}{\Big(\frac{t^2}{u}\Big)^{\frac{1}{2}}}\frac{du}{u}\\
    &\hspace{-2cm}\leq Ct\sum_{\substack{m\in \mathbb Z\\m\neq n}}\frac{|f(m)|+|f(n)|}{\Gamma (|n-m|+1)}\int_0^1\frac{du}{\sqrt{u}}\\
    &\hspace{-2cm}\leq Ct\left(\sqrt{1+|n|}\sum_{m\in \mathbb Z}\frac{|f(m)|}{\sqrt{1+|m|}}+|f(n)|\sum_{m\in \mathbb Z}\frac{1}{\Gamma (|m|+1)}\right).
\end{align*}
Therefore,
\begin{equation}\label{3.6}
\lim_{t\rightarrow 0^+}S_2(n,t)=\sum_{\substack{m\in \mathbb Z\\m\neq n}}(f(m)-f(n))\int_0^1\frac{p_u(n-m)}{u}du,\quad n\in \mathbb Z.
\end{equation}

Finally, let us show that
\begin{equation}\label{3.7}
\lim_{t\rightarrow 0^+}(S_3(n,t)+2\black{f(n)}\log t)=f(n)\left(\int_{1/4}^\infty \frac{e^{-v}}{v}dv-\int_0^{1/4}\frac{1-e^{-v}}{v}dv\right),\quad n\in \mathbb Z.
\end{equation}

Since $p_u(k)>0$, $k\in \mathbb Z$ and $u>0$, and $\displaystyle\sum_{k\in \mathbb Z}p_u(k)=1$, by using the monotone convergence theorem we can write
\begin{align*}
S_3(n,t)&=f(n)\int_0^1\frac{ {e^{-\frac{t^2}{4u}}}}{u}du\black{=f(n)\int_{t^2/4}^\infty \frac{e^{-v}}{v}dv}=f(n)\left(\int_{t^2/4}^{1/4}+\int_{1/4}^\infty\right)\frac{e^{-v}}{v}dv\\
&=f(n)\left(\int_{t^2/4}^{1/4}\frac{e^{-v}-1}{v}dv+\int_{t^2/4}^{1/4}\frac{dv}{v}+\int_{1/4}^\infty \frac{e^{-v}}{v}dv\right)\\
&=f(n)\left(\int_{t^2/4}^{1/4}\frac{e^{-v}-1}{v}dv-2\log t+\int_{1/4}^\infty \frac{e^{-v}}{v}dv\right),\quad n\in \mathbb Z\mbox{ and }t\in (0,1).
\end{align*}
Then, \eqref{3.7} is obtained.

By putting together \eqref{I1I2I3}, \eqref{3.5}, \eqref{3.6} and \eqref{3.7} and considering Theorem \ref{ThC} we deduce that
$$
\log (-\Delta_d)f(n)=\black{-}\lim_{t\rightarrow 0^+}(u_f(n,t)+2\log t)+f(n)\left(-\gamma +\int_{1/4}^\infty \frac{e^{-v}}{v}dv -\int_0^{1/4}\frac{1-e^{-v}}{v}dv\right),\quad n\in \mathbb Z,
$$
and $(v)$ es established

\vspace{0.2cm}

Finally, to prove $(iv)$, its is sufficient to note that using  \eqref{I1I2I3} and taking into account \eqref{3.5}, \eqref{3.6} and \eqref{3.7} we deduce that
$$
\lim_{t\rightarrow 0^+}\frac{u_f(n,t)}{\log t}=\lim_{t\rightarrow 0^+}\frac{S_1(n,t)+S_2(n,t)+S_3(n,t)+2f(n)\log t}{\log t}-2f(n)=-2f(n), \quad n\in \mathbb Z.
$$
 \edproof
 %%%%%%%%%%%%%%%%%%%%%%%%%%%%%%%%%%%%%%%%%%%%%%%%%%%%%%%%%%%%%%%%%%%

\bibliographystyle{acm}

\end{document}